\documentclass[final,3p,times,authoryear]{elsarticle}

\usepackage{comment}

\usepackage[english]{babel}
\usepackage{afterpage}

\usepackage{amsmath,amssymb,amsthm,amsfonts,mathtools,bbm,mathrsfs}
\usepackage{siunitx}
\usepackage{booktabs,longtable,multirow,caption}
\usepackage{pdflscape,float}

\usepackage{graphicx,xcolor}
\usepackage{subcaption}

\usepackage{algorithm}
\usepackage{algpseudocode}

\usepackage{ragged2e,setspace,titlesec,framed}
\usepackage[normalem]{ulem}
\usepackage[title]{appendix}
\usepackage{cancel}

\usepackage{hyperref}
\hypersetup{colorlinks=true, allcolors=blue}
\usepackage{placeins}

\usepackage{todonotes}
\newtheorem{proposition}{Proposition}
\newtheorem{lemma}{Lemma}

\newcommand{\indicator}{\mathbbm{1}}
\newcommand{\expect}{\mathbb{E}}
\newcommand{\prob}{\mathbb{P}}

\newcommand{\var}{\operatorname{VaR}}
\newcommand{\cvar}{\operatorname{CVaR}}

\makeatletter
\def\ps@pprintTitle{%
  \let\@oddhead\@empty
  \let\@evenhead\@empty
  \let\@oddfoot\@empty
  \let\@evenfoot\@oddfoot
}
\makeatother

\begin{document}

\begin{frontmatter}



\title{Risk-Averse Markov Decision Processes: Applications to Electricity Grid and Reservoir Management}



\author[1]{Arash Khojaste\corref{cor1}}
\ead{akhojaste@umass.edu}

\author[1]{Jonathan Pearce}
\ead{jonathanpear@umass.edu}
\author[2]{Daniela Pucci de Farias}
\ead{danielinha@gmail.com}
\author[3]{Geoffrey Pritchard}
\ead{g.pritchard@auckland.ac.nz}
\author[1]{Golbon Zakeri}
\ead{gzakeri@umass.edu}

\cortext[cor1]{Corresponding author}

\address[1]{Dept. of Mechanical and Industrial Engineering, University of Massachusetts Amherst, MA, USA}
\address[2]{School of Business, Universidad Torcuato Di Tella, Argentina}
\address[3]{Dept. of Statistics, University of Auckland, Auckland, New Zealand}

    \begin{abstract}
        This paper develops risk-averse models to support system operators in planning and operating the electricity grid under uncertainty from renewable power generation. We incorporate financial risk hedging using conditional value at risk (CVaR) within a Markov Decision Process (MDP) framework and propose efficient, exact solution methods for these models. In addition, we introduce a power reliability-oriented risk measure and present new, computationally efficient models for risk-averse grid planning and operations.
    \end{abstract}



    \begin{keyword}
        stochastic processes \sep risk \sep Markov decision processes (MDP) \sep Conditional Value at Risk (CVaR) \sep offshore wind power \sep electricity grid reliability
    \end{keyword}

\end{frontmatter}

\section{Introduction}
\label{intro}

Renewable energy sources are expected to have a high share in the future electricity grids. They have become popular due to their near-zero marginal costs and zero direct emissions. However, one major challenge with renewable energy sources is that they are intermittent, as the amount of power they generate depends on weather conditions. This intermittency can reduce the overall resilience of power systems, especially during extreme events that can disrupt energy generation and distribution \citep{Suresh2024,MacKenzie2024IowaThreats}. Because of this uncertainty, a backup/firming fleet is needed to ensure grid reliability. These backup sources can be thermal energy sources, like natural gas, or energy storage systems, like batteries. However, both options face operational limitations. For example, thermal generators can only be ramped up or down by a limited amount within a specified time frame. Similarly, storage has to comply with maximum charge and discharge rates. Knowing the behavior pattern of renewable sources greatly benefits system operators as it will help them build an operating plan for running the backup system, considering its limitations. However, in electricity grids that mainly depend on renewable energy sources, especially wind and solar, ensuring supply adequacy requires more than simply aligning installed generation capacity with peak demand and a reserve resource \citep{Stenclik}.
System operators need new tools, alongside backup energy resources, to help them balance the demand with the variability of renewable energy production \citep{Mortimer}.

Investment decisions require reliable estimates of operational costs and a near-optimal operating plan for how to run the system. This need motivates our use of the Markov Decision Process (MDP) framework. In general, optimization problems related to finding an operating plan for the ramping plant generators are mostly based on the objective of minimizing the \textit{expected cost}, which is the long-run average cost of the system. Meanwhile, tail risks can disproportionately impact (instantaneous/pointwise) cost and/or reliability and should be included in resource adequacy analysis \citep{Stenclik2}. System operators require risk-averse decision-making by choosing the optimal policy to run the backup system in a way that is not only cost-efficient but also meets predefined reliability targets; e.g., Loss-of-Load Probability \citep{Mortimer}.

In our analysis, we develop methods that reflect the cost of the system when the operator hedges against the risk.
Building upon our previous study \citep{QFR}, which explored how to run the backup system to minimize costs on average, we develop risk-averse models in this paper, hedging against the worst-case costs of a backup system.

In our model, losses in the electricity system arise from two primary sources: the expenses incurred from operating the available backup system and the costs associated with demand curtailment. We have developed a risk model that captures these costs in terms of how they translate to overall system loss. This model provides a trade-off between fuel costs and the cost of the demand that will be curtailed. Although one might assume that increasing the penalty of demand curtailment can fully capture its impacts, this approach falls short of addressing the complexity of tail risks and consecutive curtailments, which can be significantly more detrimental to end consumers. Such concerns motivate alternative approaches, which we explore through different variations of risk models in this paper. These variations consist of incorporating Conditional Value-at-Risk (CVaR) and implementing strategies to prevent consecutive demand curtailment.

Recent work has studied investment decision problems and risk-averse planning in several domains. Some examples include security and cybersecurity \citep{Akbari2025,Shojaeighadikolaei2025}, transportation \citep{Zeigham2025RiskPavement},
robotics \citep{Khass2025NBV}, finance \citep{Saghezchi2024}, and supply chain management \citep{Sobhani2019AvailabilityOptimization}.
In particular, managing risk in stochastic energy systems has been studied extensively. For example, \cite{Nasr} presents an approach that considers the uncertainties of photovoltaic power generation and demand by solving a bi-level multi-objective optimization problem using information gap decision theory in microgrids. This non-probabilistic approach does not require knowledge of the probability density functions for uncertain parameters. \cite{Yankson} does not explicitly include risk measures but adopts a planning approach that analyzes the negative impacts of power loss and shows how effective planning can mitigate them. Their approach implicitly addresses risk by introducing an energy management framework that ensures energy availability in a microgrid after natural disasters. \cite{Bakhtiari} proposes a modified Metropolis-coupled Markov chain Monte Carlo simulation for planning a renewable energy-based stand-alone microgrid. They demonstrate that the modified model effectively reduces the risk of power shortages during system operation. \cite{Mortimer} investigates risk-aware linear policy approximations in energy-limited and stochastic energy systems by incorporating CVaR. It relies on rolling forecasts in a sequential decision-making framework. Instead of using an MDP, it employs a parameter-modified cost function approximation to find an equivalent scenario that performs comparably to the stochastic methods. \cite{Lara} proposes a use-case for probabilistic forecasts by incorporating them into hour-ahead operations to enhance situational awareness through a risk-averse multi-stage stochastic program. They employ a Markovian representation of probabilistic forecasts that can capture the variability of Renewable Energy. This enables formulating a multi-stage problem while avoiding the scenario generation phase.
Some research on risk-averse sequential decision-making in energy systems has explored risk-averse Model Predictive Control (MPC) methods using CVaR, as presented in studies such as \cite{Ning} and \cite{Hans}. While these studies were inspiring, they do not directly propose or solve the specific formulation we present.

 Many real-world planning problems can be categorized into two types: ``investment'' and ``operational.'' Each decision made during the investment phase can lead to a series of potential actions in the subsequent operational phase. The operational phase generally represents a long-term horizon and is often divided into multiple stages. 
We prefer representing the operational phase through a discrete MDP, as system operators need an estimate of the cost of the long-run operating plan, which will be valuable in the investment phase.

Stochastic processes are becoming more and more prevalent in modeling complex, uncertainty-driven systems in different applications (see \cite{Farhang2025PipelineFaulting,Sobhani2025ExtendedBlockDiagram}). In particular, much research is on creating a risk-aware MDP that extends beyond the energy domain. One of the earliest works in this area, \cite{Howard}, considers the maximization of certain equivalent reward generated by an MDP with constant risk sensitivity. It utilizes value iteration for time-varying finite-horizon optimization and develops a policy iteration procedure to find the stationary policy with the highest certain equivalent gain for the infinite duration case. In more recent work, \cite{Ahmadi} formulates constrained risk-averse MDPs using a Lagrangian framework, employing coherent risk measures such as Conditional Value-at-Risk (CVaR) and Entropic Value-at-Risk to model uncertainty in decision making.
\cite{Chow} proposes algorithms for CVaR optimization in MDPs to achieve locally risk-sensitive optimal policies.
\cite{Bäuerle} explores spectral risk measures in MDPs, proposing an approach that minimizes total discounted cost using a decomposition approach. Their work extends traditional risk-sensitive MDP models by incorporating spectral risk measures as a generalization of CVaR.

Besides the formulation and the solution method, another distinguishing feature of our work is the construction of the MDP states via quantile Fourier regressions. Quantile regressions are ideal for modeling periodically varying random phenomena (\cite{Pritchard_hydro})—such as electricity demand and renewable generation, which exhibit strong daily and seasonal patterns. We allow \textit{time-inhomogeneous} Markov transition probabilities, allowing both the marginal behavior and serial dependence structure to vary periodically. In the Background Section~\ref{sec:background}, we review some of the key aspects of our MDP formulation. For a more detailed explanation, see \cite{QFR}. Building on this, Section~\ref{sec:riskaverseMDP} introduces our risk-averse decision-making framework utilizing Conditional Value-at-Risk (CVaR). It presents a bilinear program formulation along with a new solution method that efficiently produces a globally optimal solution. 
Section~\ref{sec:results} shows numerical results of our model on two applications. We conclude in Section~\ref{sec:LPvariations} by exploring alternative risk-related approaches that keep linear program formulations to improve electricity grid reliability. These methods collectively enhance system resilience while maintaining computational tractability.


\color{black}

\section{Background}
\label{sec:background}

This section provides the foundation of our MDP framework, which is explained in detail in our recent work \cite{QFR}. We begin by describing the quantile Fourier regression technique, which is used to characterize the periodic behavior of an exogenous input. This behavior will define our MDP state space. Then, we outline the structure of the MDP itself, including its characteristics: states, actions, costs, and transition dynamics. Finally, we present two motivating case studies to illustrate the application of our risk-averse approaches later in this paper.


\subsection{Markov decision process for grid reliability}

\begin{table*}[t]
\caption{Model quantities in our two case studies.}
\label{tab:model-quantities}
\centering
\renewcommand{\arraystretch}{1.15}
\setlength{\tabcolsep}{6pt}
\begin{tabular}{@{}l c l c c l c c@{}}
\toprule
& & \multicolumn{3}{c}{Example 1: Thermal backup} & \multicolumn{3}{c}{Example 2: Reservoir management} \\
\cmidrule(lr){3-5} \cmidrule(lr){6-8}
Variable & Symbol
& Description & Set & Unit
& Description & Set & Unit \\
\midrule
Temporal state
& $t \in T$
& period & $1{:}24$ & hour
& period & $1{:}54$ & week \\

Input regime state
& $r \in R$
& net demand state & $1{:}4$ & index
& inflow state & $1{:}4$ & index \\

Operational state
& $\ell \in L$
& generation level & $0{:}13$ & units
& storage level & $0{:}5000$ & MW \\
\midrule
Action / control
& $a \in A$
& ramp step & $-1{:}1$ & units
& reservoir outflow & $0{:}9$ & MW \\

Exogenous input
& $w \in W$
& net demand & $(-\infty, +\infty)$ & MW
& reservoir inflow & $0{:}47$ & MW \\
\midrule
Frequency
& $\{\tau_1, \tau_2, \ldots\}$
& diurnal \& annual & $\{1, 365\}\times 24$ & hour
& annual & $\{1\}$ & week \\
\bottomrule
\end{tabular}
\end{table*}

MDPs provide a systematic framework for choosing policies that optimize a desired objective. MDPs may have different types, including finite or infinite states, finite or infinite time horizons, and continuous or discrete processes. In this paper, we always assume a finite-state discrete-time process with an infinite (and cyclic) time horizon.


The renewable energy system itself is then modeled as a discrete Markov decision process \citep{Bertsekas_DP_Vol1, Puterman, White} comprising a collection of observable \emph{state} variables and manipulable \emph{action} variables, augmented with one-step \emph{state transition probabilities} and state-action \emph{costs}, additive over any state-action sequence. Now we define the components of our MDP:

\subsubsection{MDP variables}

\begin{itemize}
    \item An exogenous random variable \(w(t) \in W\) (``electricity net demand/water'') that quantifies the fluctuating element of a renewable energy system.
    \item Observable \emph{state} components:
          \begin{enumerate}
              \item A \emph{temporal} state \(t \in T \subset \{1, 2, \ldots , |T| \}\) that varies cyclically over its range.
              \item An environmental \emph{regime} state \(r(t) \in R \subset \{1, 2, \ldots, |R|\}\) indexing elements of a partition \(\cup_{r \in R} W_r(t) = W(t)\) of the quantiles of the exogenous signal \(w\).
              \item An operational state \(\ell(t) \in L\) (``level'') of the generator itself.
          \end{enumerate}
    It will be convenient to pack the state variables as \((t, s) \coloneqq(t, (\ell, r)) \in T \times (L \times R) \eqqcolon T \times S\).
    \item A manipulable \emph{action} (control) variable \(a \in A\).
    \item The \textit{decision time points} are the same as time $t \in T$ where $t=1,2,\ldots,|T|$.
    
    \item There is an \textit{immediate cost} $c_{tsa}$ associated with taking action $a$ at the state $s$ at the corresponding time $t$.
    
    \item The $p_{s^\prime \mid t s a}$ are the \textit{transition probabilities} of going from state $s$ to state $s\prime$ with action $a$ at time $t$.
\end{itemize}


\subsubsection{A model of random inputs}

We assume \(w\) is a cyclostationary process, as is natural for wind/solar potential or hydrologic inflow.
A fixed set of quantile curves \(\{w_1(t), \ldots, w_{|R|-1}(t)\}\) of \(w(t)\) is assumed periodic and well approximated from historical time series via Fourier basis regression:
\begin{align*}
    w_r(t \mid B_\alpha) \coloneqq {} & \phi_W(t)\cdot B_{r}
\end{align*}
Here, \(\phi_W\) is a basis of Fourier polynomials, whose order and frequencies reflect the assumed periodicity of the system. The regression coefficients \(B_r\) are determined for a sequence of preselected levels via quantile regression \citep{QFR, Koenker}, each of which entails the solution of a linear program.

The inter-quantile intervals \(W_r(t) \coloneqq [w_{r-1}(t), w_r(t)]\) with \(r \in \{0, \ldots, |R|\}\) define the input regimes with \(w_0(t) \coloneqq 0\) and \(w_{|R|} \coloneqq \infty\).

The conditional distribution 
\[
    p_{w\mid rt} = \prob(w \mid r, t) = \prob(w \in W_r \mid t)
\]
of the exogenous input $w$ [MW], can be approximated from the time series data via histogram regression.





The regime series \(r(t)\) is modeled as a Markov chain with time-dependent, periodic state-transition probabilities:
\begin{align*}
    p_{r^\prime\mid rt}(\vec{\gamma}) \coloneqq {} & \phi_R(t)\cdot \gamma \approx \prob(r^\prime \mid r, t),
\end{align*}
where \(\phi_R\) is another application-specific Fourier basis.
The regression parameters \(\vec{\gamma}\)
are determined via maximization of the associated log-likelihood function, subject to (convex quadratic) non-negativity constraints and (linear) normalization constraints:
\begin{align*}
    \arg\max_{\vec{\gamma}} ~ \sum_{(t,r,r^\prime) \in \text{data}} \log p_{r^\prime \mid rt} (\vec{\gamma}) {} &
    \\
    \sum_{r^\prime} p_{r^\prime\mid rt}(\vec{\gamma}) = {}                                                     & 1 \quad \forall t, r
    \\
    p_{r^\prime\mid rt}(\vec{\gamma}) \in {}                                                                   & [0, 1] \quad \forall t, r, r^\prime
\end{align*}
This problem reduces to a second-order cone-constrained convex program \citep{Lobo-et-al:1998:SOCP}.

\subsubsection{State transition model}


The state variables evolve via a deterministic model \(\psi\), driven by the stochastic input \(w\) and control action \(a\):
\begin{align*}
    t^\prime = {}    & t \bmod |T| + 1
    \\
    r^\prime \sim {} & \prob_R(\bullet \mid t, r)
    \\
    \ell^\prime = {} & \psi(\ell, a, w)
    \\
    w^\prime \sim {} & \prob_W(\bullet \mid t, r)
\end{align*}

We describe examples of \(\phi_W, \phi_R, \psi\) in Sections~\ref{sec:thermal-backup-model}--\ref{sec:reservoir-model} for two concrete case studies. Here $t \bmod |T|$ denotes the remainder after dividing $t$ by $|T|$, so that the time index cycles through $\{1,\dots,|T|\}$ via $t' = (t \bmod |T|) + 1$.

\subsection{Minimum expected cost operating policy}

\color{black}

We now introduce the linear program formulation MDPLP~\eqref{eq:MDPLP} for finding the optimal decision. To avoid repetition, we define the feasible set of state-action probabilities as:
\begin{align}
    X \coloneqq {} &
    \left\{ x \in \Delta \;\middle|\;
    \begin{aligned}
        \sum_{a} x_{t^\prime sa} = \sum_{s^\prime, a} x_{ts^\prime a} p_{s^\prime \mid t s a}  \quad \forall t,  s
    \end{aligned}
    \right\}
    \label{eq:feasible_set}
\end{align}
Here,
\begin{align}
    \Delta \coloneqq {} &
    \left\{ x \geq 0 \;\middle|\;
    \begin{aligned}
        \sum_{s, a} x_{tsa} = 1 \quad \forall t
    \end{aligned}
    \right\}
    \nonumber 
\end{align}
is the set of \(T\)-indexed probability distributions over \(S \times A\); \(X\) is the subset that is compatible with the one-step state-transition probabilities on \(S\).


Each state-action distribution \(x \in X\) corresponds to a feasible cyclostationary operating policy
\begin{align*}
    \prob(s, a \mid t) = x_{tsa}, 
\end{align*}
and each \(x\) confers a distribution on the set of stage costs \(\{c_{tsa}: t \in T, s \in S, a \in A\}\).


\begin{align}
    \text{MDPLP} \quad
    \min_{x}~\left\{\sum_{t, s, a} c_{tsa} \frac{x_{tsa}}{|T|}: x \in X \right\}
    \label{eq:MDPLP}
\end{align}

Here $x_{tsa}$, the only decision variables are the stationary unconditional probabilities that the system is in state $s$ and action $a$ is taken at time $t$.

In \citep[\S19.3]{HillierLieberman}, the authors mention that for any optimal vertex solution of MDPLP~\eqref{eq:MDPLP} at each state $s \in S$ at any given time $t$, only \textit{one} action $a$ can have strictly positive value. Consequently, the MDPLP~\eqref{eq:MDPLP} will deliver a deterministic policy.

\begin{proposition}
    \label{prop:deterministic}
    There exists an optimal policy found by solving the LP formulation \eqref{eq:MDPLP} which is \textit{deterministic} rather than randomized. \citep{HillierLieberman}
\end{proposition}

The MDPLP~\eqref{eq:MDPLP} formulation is explained in detail in \citep{HillierLieberman} and \citep{White}, and it identifies the optimal decision policy for a Markov decision process based on minimizing total ``average cost''. The decision policy involves a series of actions undertaken in each state. The primary goal of this model is to minimize the expected costs in the steady state, thereby ensuring that the chosen policy yields optimal average costs over time. 
As a planning tool, the transition probabilities in the formulation described herein rely entirely on historical data---it does not exploit rolling forecasts (cf. \cite{Mortimer} and \cite{Ghadimi})---so cannot be expected to provide the best online operating plan. However, the procedure is ideal for estimating operating costs for use in long-term infrastructure/capacity planning. One recent example of this model can be found in our recent work on optimizing cooling policy of data centers in \cite{Khojaste2025Cooling}.
First, we will use this as our base model in the context of grid reliability as a Markov decision process. In the following sections, we will introduce different methods to include risk measures in this base model.

\subsection{Model set-up: thermal backup of offshore wind}\label{sec:thermal-backup-model}

\begin{table}[ht]
    \centering
    \begin{tabular}{llc}
        Name                   & Description                & Value        \\ \hline
        \texttt{base\_ramp}    & base generation capacity   & 6\,GW        \\
        \texttt{penalty\_cost} & cost of demand curtailment & \$3.0/MWh    \\
        \texttt{fuel\_cost}    & cost of thermal generation & \$0.1/MWh    \\
        \texttt{ramp\_rate}    & unit generation capacity   & 1.5\,MW/unit \\
    \end{tabular}
    \caption{Parameters of the thermal backup case study.}
\end{table}

In a grid with an intermittent renewable source, renewable generation (with zero marginal cost) is dispatched first, while backup thermal resources are dispatched to meet the residual demand. We refer to this residual demand as ``net demand.'' Because the renewable source is intermittent, the system will face a random level of net demand. If we distinguish $|R|$ different levels of net demand, this random variation can be represented by a discrete process moving between those levels. By analyzing the historical data on electricity demand and considering a renewable energy generation based on the corresponding area, we achieve the parameters to fit a Markov chain to the random variable of net demand.

This study extends \cite{QFR} by incorporating a range of risk preferences. As in \citep{QFR}, wind represents the intermittent renewable source, and thermal generation serves as a backup for our model. We used hourly electricity demand data from 2006 to 2020 from FERC \citep{FERC}. The wind power data was derived from hourly wind speed data from NOAA's Buoy 44025 \citep{station44025} combined with the power generation characteristics of the IEA 15\,MW offshore wind turbine \citep{IEA15}. The periodic nature of the net demand was captured using quantile Fourier regression to effectively model time-dependent behaviors. Figure \ref{fig:net_demand_quantiles_15_yearsnew} shows quantiles 0.75, 0.5, and 0.25 fitted the net demand data from 2006 to 2020. Using the three fitted quantiles, we were able to model the serial dependence structure of the net demand level as a four-state Markov chain, assuming one as the lowest net demand level and four as the highest.

\begin{figure} [ht!]
    \centering
    \begin{minipage}{.5\textwidth}
        \centering
        \includegraphics[width=1\textwidth]{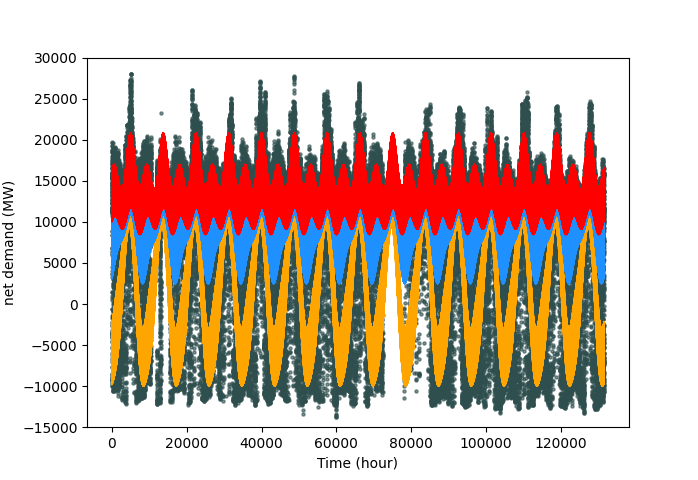}
    \end{minipage}%
    \begin{minipage}{.5\textwidth}
        \centering
        \includegraphics[width=1\textwidth]{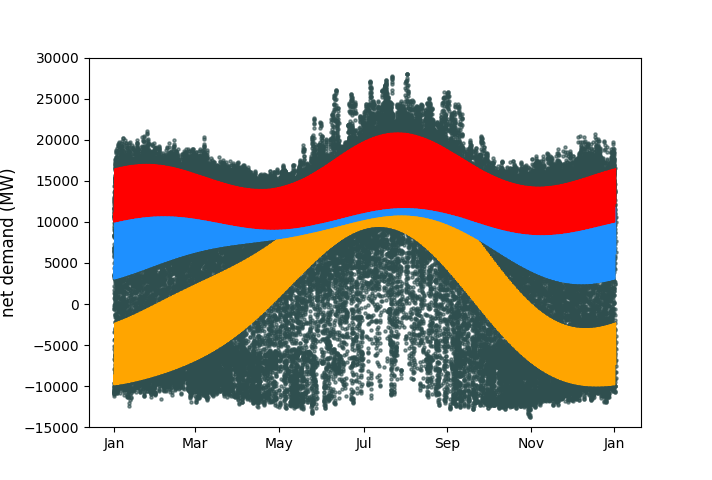}
    \end{minipage}
    \caption{Quantile Fourier regression with both annual and daily periodicity, fitted to 15 years of net demand data with two regressors.
        Left panel: normal plot. Right panel: phase-folded plot.}
    \label{fig:net_demand_quantiles_15_yearsnew}
\end{figure}

To make a time-dependent plan for our second model, we introduced a \textit{maximum likelihood} model where we allow all sixteen transition probabilities to vary annually as simple sinusoids. A detailed explanation can be found at \cite{QFR}.

In our examples, we assume having $13$ Combined-Cycle Gas Turbine (CCGT) power plants with a total of 21 \,GW generation capacity as the backup system. The allowable actions for the model at each state are ``ramp-up,'' ``stay at the same operating level,'' and ``ramp-down.'' The increase or decrease in generation by actions ramping up or down is assumed to be 1.5 \,GW in total, while we have 6 \,GW as base generation.

\subsubsection{Variables and simplifications}

The system variables are summarized in Table \ref{tab:model-quantities}.

\subsubsection{State transition model}

\begin{align*}
    \ell^\prime = {} & \operatorname{clamp}(\ell + a, \min(L), \max(L))
\end{align*}

\subsubsection{Stage costs}

\begin{align*}
    \text{generation} \coloneqq {} & (\texttt{base\_ramp} + \ell ) \times \texttt{ramp\_rate}
    \\
    \text{Curtailed Demand = $CD_{t(lr)}$} \coloneqq {}  & \max(w_{tr} - \text{generation}, 0)
    \\
    c_{t(\ell r)a}
    \coloneqq {}                   &
    \text{generation} \times \texttt{fuel\_cost}                                             \\
    {}                             & + \text{$CD_{t(lr)}$} \times \texttt{penalty\_cost}
\end{align*}

\subsection{Model set-up: hydropower reservoir management}\label{sec:reservoir-model}

In this section, we apply our methods to another example. We adopt the setting of the hydropower reservoir management problem explained in detail in the \cite{QFR}. This setting models the operation of a hydropower system subject to uncertain seasonal inflows and limited storage capacity. The objective is to satisfy constant electricity demand using both hydro and non-hydro generation sources while minimizing an objective under inflow uncertainty. For a detailed description of how this reservoir management setting can be used to construct baseline hydropower generation offer curves, we refer the reader to \cite{Pearce2025}.


System demand is fixed and the marginal cost of generation from the variable renewable resource is negligible, but net demand is now determined by the operating policy to balance opportunity cost. Since each planning period is now one week (rather than one hour), ramping constraints on the thermal generators are now neglected.

\begin{table}[ht]
    \centering
    \begin{tabular}{llc}
        Name                   & Description             & Value                      \\ \hline
        \texttt{min\_outflow}  & run-of-river generation & 500\,MW                    \\
        \texttt{capacity}      & storage capacity        & 500\,MW \\
        \texttt{load}          & fixed load              & 1400\,MW                   \\
        \texttt{fuel\_cost}    & thermal fuel price      & \$50/MWh                   \\
        \texttt{penalty\_cost} & load curtailment price  & \$1000/MWh
    \end{tabular}
    \caption{Parameter values of the reservoir management case study.}
    \label{tab:parame}
\end{table}

\subsubsection{Variables and simplifications}

The system variables are summarized in Table \ref{tab:model-quantities}.

For simplicity, we formulate the model for a single reservoir; however, the model is easily generalizable to multiple reservoirs, and the results obtained in this study will (appropriately) extend to that case. The electrical load [MW] and cost of thermal generation [\$/MWh] are assumed constant when deriving the operating policy. When evaluating the policy through simulation, both may vary.


\subsubsection{State transition model}

\begin{align*}
    \text{released} \coloneqq {}  & \min(\ell + w_{tr}, \texttt{min\_outflow} + a)
    \\
    \ell^\prime = {}              & \operatorname{clamp}(\ell + w_{tr} - \text{released}, \min(L), \max(L))
    \\
    \prob_S(s^\prime \mid s, a) = {} & \prob_W(w \mid r^\prime, t) \, \prob_R(r^\prime \mid r, t)
    \quad 
\end{align*}

\subsubsection{Stage costs}

While water discharge, itself, has zero marginal cost, \(a\) determines the net demand that must be met with thermal generation (nonzero marginal cost) and demand curtailment (punitive cost).

\begin{align*}
    \text{released} \coloneqq {}          & \min(\ell + w_{tr}, \texttt{min\_outflow} + a)
    \\
    \text{hydro\_dispatch} \coloneqq {}   & \max(\text{released}, \texttt{min\_outflow})
    \\
    \text{thermal\_dispatch} \coloneqq {} & \texttt{load} - \text{hydro\_dispatch}
    \\
    \text{lost\_load} \coloneqq {}        & \max(\texttt{min\_outflow} - \text{released}, 0)
    \\
    c_{t(\ell r)a} = {}                   & \text{thermal\_dispatch} \times \texttt{fuel\_cost} \\
    {}                                    & + \text{lost\_load} \times \texttt{penalty\_cost}
\end{align*}

So far, we have had the optimal operating plan of the system coming from an MDP, which is based on the average costs, as we are minimizing the expected costs of the system based on problem \eqref{eq:MDPLP}. The following sections will explain how we can incorporate risk aversion into our decision-making process.

\section{Risk-Averse Markov Decision Process}
\label{sec:riskaverseMDP}
In some applications, achieving the optimal operating plan based only on the average cost is not sufficient, as the decision-maker is not necessarily risk-neutral. Therefore, it is necessary to consider the decision-maker's risk aversion. This involves taking into account the tail behavior of the cost function distribution. One way to implement this is to incorporate a coherent risk measure in the original formulation.
\cite{Artzner} characterized a set of four natural properties desirable for a risk measure, called coherent risk measures, and which has obtained widespread acceptance in various fields, including finance and operations research. Conditional Value-at-Risk (CVaR) is one of the coherent risk measures that has received significant attention in decision-making problems, such as Markov decision processes (MDPs)\citep{Chow}. We use CVaR as a risk measure to hedge against the worst outcomes of the cost function. This can lead to having a \textit{mean-risk} model that ensures robust decision-making under uncertainty.

\subsection{A review of Conditional Value-at-Risk}

The \(\beta\)-quantile (Value-at-Risk) of a random variable \(C\) representing the state-action cost is:
\begin{align*}
    \eta \coloneqq \var(C, \beta) \coloneqq \min \{\alpha: \beta \leq \prob_C(C \leq \alpha)\}
\end{align*}
The corresponding \emph{Conditional Value-at-Risk} at probability level \(\beta\) is defined as the mean of the \(\beta\)-tail distribution:
\begin{align*}
    f \coloneqq \cvar(C, \beta) \coloneqq
    \expect_C\left(\frac{\indicator_{\eta \leq C}}{1 - \beta}\, C\right)
\end{align*}

If \(C\) has a fixed distribution \(\prob(C = c_i) = p_i\), \cite{Rockafellar2000} explain that \(f\) can be computed via a linear program:
\begin{align}
    f
    = {} & \min_{\eta, ~ u \geq 0} ~ \eta + \frac{1}{1 - \beta} \sum_{i} p_i ( c_i - \eta)^+
    \label{eqn:cvar-convex-decoupled}
    \\
    = {} & \min_{\eta, ~ u \geq 0} ~ \left\{\eta + \frac{1}{1 - \beta} \sum_{i} p_i u_i: u_i \geq c_i - \eta \quad \forall i\right\}
    \label{eqn:cvar-linear-coupled}
\end{align}
when \(p\) is known, these equivalent programs are convex: \eqref{eqn:cvar-linear-coupled} is linear, while the feasible set in \eqref{eqn:cvar-convex-decoupled} is a simple product.

\subsection{Interpretation of CVaR}
\label{Interpretation_of_CVaR}
Before turning to the case where \(p\) is a decision variable, consider two reformulations for interpretation. Both are well established, but perhaps not very widely known.

Introducing two non-negative slack variables \(v \geq 0\) and \(u \geq 0\) in the constraint of \eqref{eqn:cvar-linear-coupled}
      \begin{align*}
          c_i - \eta \eqqcolon {} & u_i - v_i,
      \end{align*}
      and using the identity
      \begin{align*}
          \eta \equiv {} & 1\eta \equiv \left(\sum_i p_i\right) \eta \equiv \sum_i p_i \eta,
      \end{align*}
      the objective function of \eqref{eqn:cvar-linear-coupled} can be rewritten:
      \begin{align*}
          1 \eta + \frac{1}{1 - \beta} \sum_{i} p_i u_i
          \equiv {} & \sum_i p_i \left(\eta + \frac{1}{1 - \beta} u_i\right)
          \\
          \equiv {} & \sum_i p_i \left(c_i - u_i + v_i + \frac{1}{1 - \beta} u_i\right)
          \\
          \equiv {} & \sum_i
          \overbrace{\left(c_i + \frac{\beta u_i + (1 - \beta) v_i}{1 - \beta}\right)}^{\bar{c}_i \coloneqq} ~ p_i
      \end{align*}
      Hence, \eqref{eqn:cvar-linear-coupled} has an equivalent form:
      \begin{align}
          f = \min_{\eta, u \geq 0, v \geq 0} ~ \left\{
          \sum_i \left(c_i + \frac{\beta u_i + (1 - \beta) v_i}{1 - \beta}\right) p_i:
          u_i - v_i = c_i - \eta \quad \forall i\right\}
          \label{eqn:cvar-over_under}
      \end{align}
      The modified objective is another expectation: The original costs \(c_i\) exceeding quantile \(\eta\), are penalized in proportion with \(\beta\).



\color{black}

\subsection{The Bilinear Program Formulation}

In our context, the probabilities \(p\) are decision variables to be determined, so \eqref{eqn:cvar-linear-coupled} and equivalently \eqref{eqn:cvar-over_under} become nonconvex bilinear programs. Using \eqref{eqn:cvar-over_under}, we define the risk-averse version of the MDPLP~\eqref{eq:MDPLP} as:

\begin{equation}
\begin{array}{lll}
\text{MDPBL}  \displaystyle \min_{x,\eta} & \frac{1}{|T|} \sum_{t \in T} \sum_{s \in S} \sum_{a \in A} c_{tsa} x_{tsa} + \lambda x_{tsa} \left( \beta (c_{tsa}-\eta)^{+} + (1-\beta) (\eta - c_{tsa})^{+} \right) &  \\

            \text{s.t.} & x \in X\\
                        &  \eta \text{  is free.}
\end{array}
\label{eq:MDPBL}
\end{equation}
where \( x_{tsa} \) and \( \eta \) are decision variables, and \( \lambda \) and \( \beta \) are constants, and \( (a)^{+} \) is the \textbf{positive-part function} of $a$ as \( (a)^{+} = \max\{a,0\} \).

Here, $\lambda$ is a coefficient that determines the weight of the risk component in this mean-risk formulation, which is $\lambda = \frac{1}{1 - \beta}$. \citep{Rockafellar2000}

We can see that this optimization problem is no longer a linear program problem; instead, it takes the form of a bilinear problem. This bilinear problem is challenging to solve even for a small-scale problem because as it is non-convex. Generally, two different but complementary approaches in non-convex programming exist: global and local. Global approaches, such as spatial branch-and-bound, can guarantee the global optimality of the solutions but are time-consuming, particularly for large-scale problems. On the other hand, local approaches are much faster but can only reach local solutions that are not necessarily global optima \citep{LePham}.

This bilinear problem takes much time to solve using the ``Gurobi 12'' solver, which employs the spatial branch-and-bound method. Even for a small problem with only $24$ time steps ($|T|=24$) which leads to the same plan for every day, it takes hours to prove optimality, as the optimality gap is still more than 17 percent after one day of run-time. As the problem size increases, the gap between the best bound and the current objective function grows, and the run-time rises significantly. This makes the problem prohibitively large for our problem, which has a unique operational plan for every day of the year ($|T|=8760$). In the next section, we briefly discuss local approaches for this problem; however, our primary focus is an exact method that solves it globally.

\subsection{Difference-of-Convex program to Reach a Local Optimum}
One way to solve the non-convex problems more efficiently is to reformulate them as Difference-of-Convex (DC) programs. Doing so allows us to use various algorithms commonly used in the DC programs. Before doing that, for simplification, we use two auxiliary variables, $U_{tsa}$ and $V_{tsa}$, to represent $(c_{tsa}-\eta)^{+}$ and $(\eta - c_{tsa})^{+}$ respectively. This can be ensured by adding the constraint $ \eta + U_{tsa} - V_{tsa} - c_{tsa} = 0$.

\begin{equation}
\begin{array}{lll}
\text{MDPBL } \displaystyle \min_{x,\eta,U,V} & \frac{1}{|T|} \sum_{t \in T} \sum_{s \in S} \sum_{a \in A} c_{tsa} x_{tsa} + \lambda x_{tsa} (\beta U_{tsa} + (1-\beta) V_{tsa}) &  \\
                                \text{s.t.} & x \in X\\
                                & \eta + U_{tsa} - V_{tsa} - c_{tsa} = 0 & \forall t, \forall s, \forall a \\
                       & U_{tsa} \geq 0  & \forall t, \forall s, \forall a\\
                       & V_{tsa} \geq 0  & \forall t, \forall s, \forall a\\
                       &  \eta \text{  free.}
\end{array}
\label{eq:MDPBLUW}
\end{equation}

There are different ways to make this problem a DC program.  One way to do this is using the simple fact that $xy=\frac{1}{4}[(x+y)^2-(x-y)^2]$. Accordingly, the objective function of \ref{eq:MDPBL} can be reformulated as:

\begin{equation}
\frac{1}{|T|} \sum_{t \in T} \sum_{s \in S} \sum_{a \in A} c_{tsa} x_{tsa} + \frac{\lambda}{4} (\beta [(x_{tsa}+U_{tsa})^2-(x_{tsa}-U_{tsa})^2] + (1-\beta) [(x_{tsa}+W_{tsa})^2-(x_{tsa}-W_{tsa})^2] )
\label{eq:DCA}
\end{equation}



After expanding \eqref{eq:DCA}, we can have a difference of two convex functions, as we have $f$ and $g$ below, where both are convex functions. The objective function is then expressed as $f-g$ subject to the constraints at \eqref{eq:MDPBL}:

\begin{equation}
f= \frac{1}{|T|} \sum_{t \in T} \sum_{s \in S} \sum_{a \in A} c_{tsa} x_{tsa} + \frac{\lambda}{4} \beta (x_{tsa}+U_{tsa})^2+\frac{\lambda}{4}(1-\beta) (x_{tsa}+W_{tsa})^2 \\ 
\end{equation}
\begin{equation}
g= \frac{1}{|T|} \sum_{t \in T} \sum_{s \in S} \sum_{a \in A} \frac{\lambda}{4} \beta (x_{tsa}-U_{tsa})^2 + \frac{\lambda}{4}(1-\beta) (x_{tsa}-W_{tsa})^2
\end{equation}


\begin{equation}
\begin{array}{lll}
\text{MDPDC min } & f-g \\
                                \text{s.t.} & \text{constraints of} \eqref{eq:MDPBLUW}
\end{array}
\label{eq:MDPDC}
\end{equation}

We can apply various algorithms to solve this DC problem at this stage. One method we can use is the Difference Convex Algorithm (DCA) introduced by Pham Dinh and Le Thi in their early work \citep{PhamDinh}. This algorithm is a local search method that does not guarantee optimality. Other local solvers, such as CONOPT, can also solve medium-scale problems. Since we present a method to solve the problem to global optimality in the next section, we will not examine results from local solutions; instead, we leave this exploration to the reader.


\subsection{An Exact Search Method to Reach the Global Optimum}

Due to its non-convex nature, solving the bilinear program \eqref{eq:MDPBL} using global methods such as spatial branch-and-bound becomes computationally prohibitive. To address this computational challenge, we propose a structured search approach to find global solutions efficiently.

To solve this problem, we identified $\eta$ as a \textit{complicating variable}. If we fix $\eta$, the original non-convex problem \eqref{eq:MDPBL} (denoted as the master problem) becomes a linear program \eqref{eq:subMDPBL} (subproblem), which can be solved efficiently. 

\begin{equation}
\begin{array}{lll}
\text{subMDPBL} \displaystyle \min_{x} & \frac{1}{|T|} \sum_{t \in T} \sum_{s \in S} \sum_{a \in A} c_{tsa} x_{tsa} + \lambda x_{tsa} \left( \beta (c_{tsa}-\eta)^{+} + (1-\beta) (\eta - c_{tsa})^{+} \right) &  \\
            \text{s.t.} & x \in X\\
\end{array}
\label{eq:subMDPBL}
\end{equation}

It is worth mentioning that the difference between the subproblem \eqref{eq:subMDPBL} and the master problem \eqref{eq:MDPBL} is that in the subproblem $\eta$ is a parameter rather than a decision variable. Similar to the MDPLP case \ref{eq:MDPLP}, the LP formulation for the subMDPBL \ref{eq:subMDPBL} also yields a solution with deterministic policies as mentioned in the following Proposition \ref{prop:deterministic2}.

\begin{proposition}
\label{prop:deterministic2}
There exists an optimal policy found by solving the LP formulation for the subMDPBL that is \textit{deterministic} rather than randomized.
\end{proposition}

By solving this subproblem for a fixed $\eta$, we obtain $f^{*}(\eta)= min f(x|\eta)$. Ideally, we would like to assert that \( f^{*}(\eta) \) is convex (or at least quasi-convex) in \( \eta \), so that, straightforward to solve, but that is not always the case. The convexity of \( f^{*}(\eta) \) in \( \eta \) depends on how we define \( c_{tsa} \).
In our model, because $c_{tsa}$ contains nonlinear terms (see Section \ref{sec:background}), therefore, $f^{*}(\eta)$ can inherit non‑convexity. This means that $f^{*}(\eta)$ may exhibit non-convex behavior with respect to $\eta$ (see Figures \ref{fig:nonconvexity}).

\begin{figure}[H]
\centering
\begin{minipage}{.5\textwidth}
\centering
  \includegraphics[width=1\textwidth]{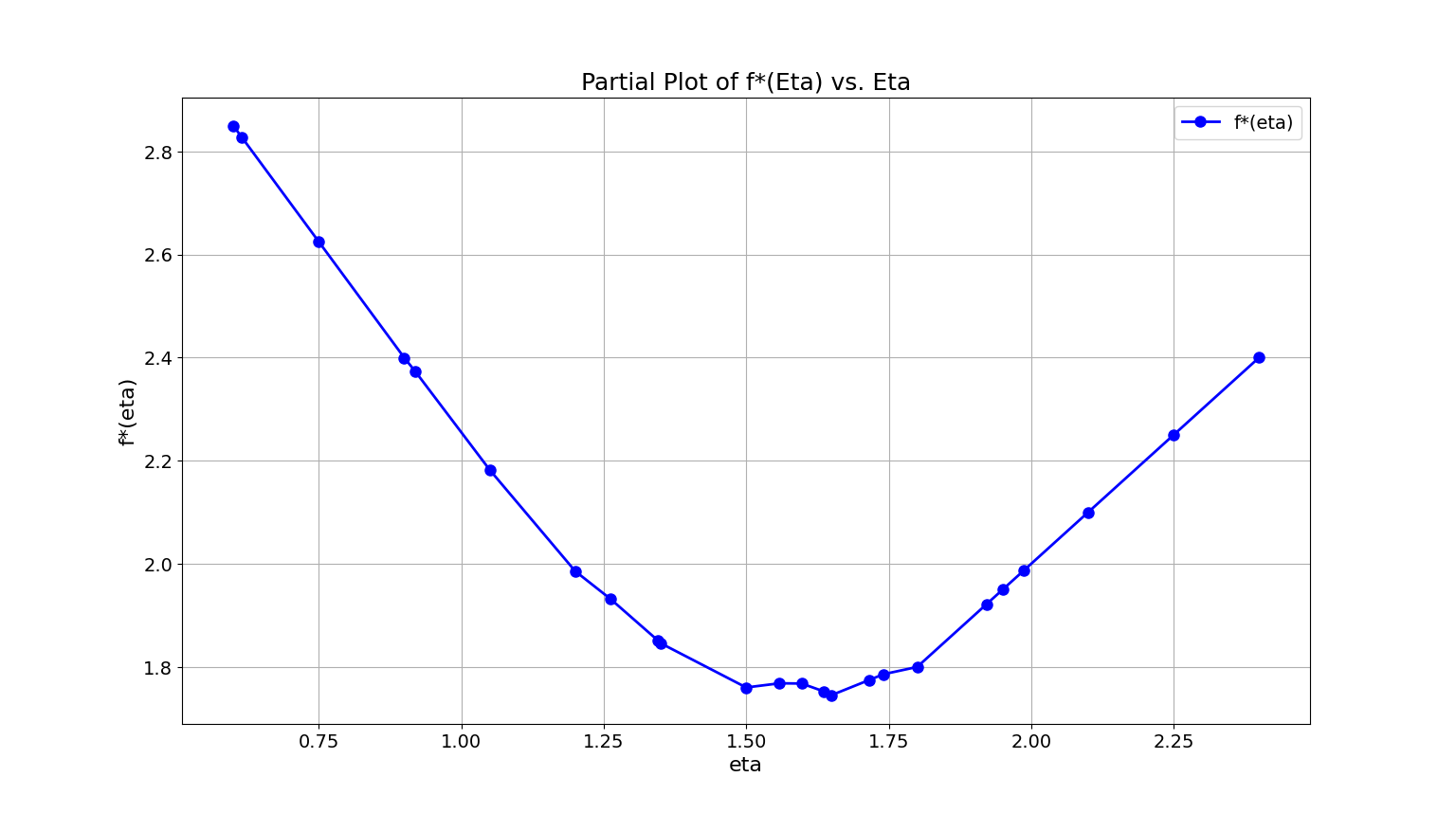}
\end{minipage}%
\begin{minipage}{.5\textwidth}
\centering
  \includegraphics[width=1\textwidth]{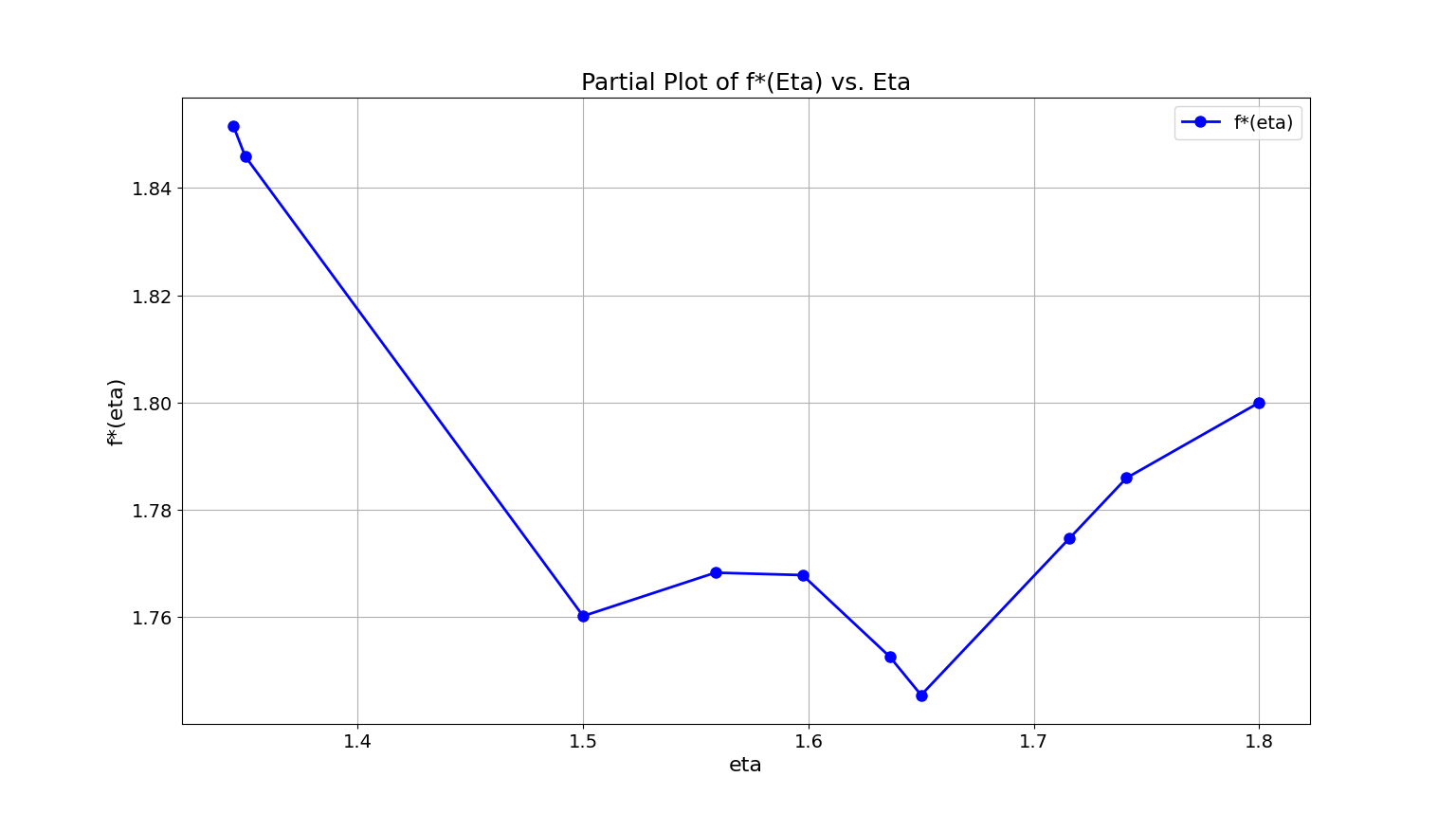}
\end{minipage}

\caption{Non-convexity that can happen in some cases for \( f^{*}(\eta) \). Left panel: A zoomed-out plot. Right panel: A zoomed-in plot.}
\label{fig:nonconvexity}
\end{figure}
It is important to recognize that the LP formulation of the MDP looks at a distribution. This distribution is finite and discrete, with \( c_{tsa} \) as its atoms, and the LP seeks the optimal \( \eta \), which is a quantile of the \( c_{tsa} \) distribution. The LP sets the probabilities \( x_{tsa} \), but in the end, the \( x_{tsa} \)s are simply the weights, even though they comply with an extra condition, which is the Markov constraint (the third constraint in \eqref{eq:MDPBL}). This interpretation motivates Proposition \ref{prop:eta=c} below.

\begin{proposition}
There exists an optimal solution $(x^*,\eta^*)$ such that $\eta^* = c_{tsa}$ for some $t,s,a$.
\label{prop:eta=c}
\end{proposition}

\begin{proof}
Let us denote the objective function by
\[
f(x,\eta) = \frac{1}{|T|} \sum_{t \in T} \sum_{s \in S} \sum_{a \in A} c_{tsa} x_{tsa} + \lambda x_{tsa} \left( \beta (c_{tsa}-\eta)^{+} + (1-\beta) (\eta - c_{tsa})^{+} \right).
\]
Then, for all feasible $(x,\eta)$ such that $\eta \geq \max_{tsa} c_{tsa}$, we have:
\[
\begin{aligned}
f(x,\eta) &= \frac{1}{|T|} \sum_{t \in T} \sum_{s \in S} \sum_{a \in A} c_{tsa} x_{tsa} + \lambda x_{tsa} (1-\beta) (\eta - c_{tsa}) \\
&= \lambda (1-\beta) \eta \cdot \frac{1}{|T|} \sum_{t \in T} \sum_{s \in S} \sum_{a \in A} x_{tsa} + \frac{1}{|T|} \sum_{t \in T} \sum_{s \in S} \sum_{a \in A} c_{tsa} x_{tsa} (1-\lambda (1-\beta)) \\
&= \lambda (1-\beta) \eta + \frac{1}{|T|} \sum_{t \in T} \sum_{s \in S} \sum_{a \in A} c_{tsa} x_{tsa} (1-\lambda (1-\beta)).
\end{aligned}
\]

In the last equality, we have used the fact that for all feasible $x$, $\sum_{sa} x_{tsa} = 1 ~\forall t$.

This implies that $f(x,\eta)$ is linear and increasing in $\eta$ for all $\eta \geq \max_{tsa} c_{tsa}$, therefore any optimal solution satisfies $\eta^*\leq \max_{tsa} c_{tsa}$. Likewise, for all feasible $(x,\eta)$ with $\eta \leq \min_{tsa} c_{tsa}$, we have
\[
f(x,\eta) = -\lambda \beta \eta + \frac{1}{|T|} \sum_{t \in T} \sum_{s \in S} \sum_{a \in A} c_{tsa} x_{tsa} (1+\lambda\beta).
\]

This implies that $f(x,\eta)$ is linear and decreasing in $\eta$ for all $\eta \leq \min_{tsa} c_{tsa}$, therefore $\eta^*\geq \min_{tsa} c_{tsa}$.

We conclude that, if $(x^*, \eta^*)$ is an optimal solution, we must have
\[
\min_{tsa} c_{tsa} \leq \eta^* \leq \max_{tsa} c_{tsa}.
\]

Now, let $(x^*, \eta^*)$ be an optimal solution, and suppose that $\eta^* \neq c_{tsa} ~\forall t,s,a$. Let
\[
c_0 = \max_{t,s,a:~ c_{tsa} \leq \eta^*} c_{tsa}, \quad
c_1 = \min_{t,s,a:~ c_{tsa} \geq \eta^*} c_{tsa}.
\]

Note that, since we are supposing that $\eta^* \neq c_{tsa} ~\forall t,s,a$, we have $c_0 < \eta^* < c_1$.

Now, for all $c_0 \leq \eta \leq c_1$, we have
\[
\begin{aligned}
f(x^*, \eta) &= \frac{1}{|T|} \sum_{t, s, a} c_{tsa} x^*_{tsa}
+ \frac{\lambda \beta}{|T|} \sum_{t, s, a: ~ c_{tsa} \geq c_1} x^*_{tsa} (c_{tsa}-\eta) \\
&\quad + \frac{\lambda (1-\beta)}{|T|} \sum_{t, s, a: ~ c_{tsa} \leq c_0} x^*_{tsa} (\eta - c_{tsa}) \\
&= \eta \left[\frac{\lambda (1-\beta)}{|T|} \sum_{t, s, a: ~ c_{tsa} \leq c_0} x^*_{tsa}
- \frac{\lambda \beta}{|T|} \sum_{t, s, a: ~ c_{tsa} \geq c_1} x^*_{tsa} \right] \\
&\quad + \frac{1}{|T|} \sum_{t, s, a} c_{tsa} x^*_{tsa}
+ \frac{\lambda \beta}{|T|} \sum_{t, s, a: ~ c_{tsa} \geq c_1} x^*_{tsa} c_{tsa}
- \frac{\lambda (1-\beta)}{|T|} \sum_{t, s, a: ~ c_{tsa} \leq c_0} x^*_{tsa} c_{tsa}.
\end{aligned}
\]

If
\[
\frac{\lambda (1-\beta)}{|T|} \sum_{t, s, a: ~ c_{tsa} \leq c_0} x^*_{tsa}
- \frac{\lambda \beta}{|T|} \sum_{t, s, a: ~ c_{tsa} \geq c_1} x^*_{tsa} > 0
\]
and $\eta^* > c_0$, then $f(x^*, c_0)< f(x^*,\eta^*)$, which is a contradiction.

Likewise, if this term is negative and $\eta^* < c_1$, then $f(x^*, c_1)< f(x^*,\eta^*)$, which is also a contradiction.

Finally, if this term is zero, then $f(x^*, c_0) = f(x^*, \eta^*) = f(x^*, c_1)$, which implies that $(x^*, c_0)$ and $(x^*, c_1)$ are also optimal solutions.
\end{proof}


Hence, the optimal threshold must coincide with one of the cost atoms \( c_{tsa} \). Since \( \eta \) is one of the \( c_{tsa} \) values, testing them all will reveal the optimal \( \eta \) with minimum \( f^{*}(\eta) \).
Without the Markov constraint, one could have easily achieved all \( x_{tsa} \) by concentrating all of the weights on states with \( c_{tsa} < \eta \). However, the Markov balance constraint ensures that all time points are connected through valid transitions.

This search for the optimal \( \eta \) that minimizes \( f^{*}(\eta) \) naturally takes place over all of the \( c_{tsa} \) values. However, the next lemma can help us significantly narrow down the search space. To explain that, first we introduce an alternative but equivalent formulation of our problem in the following proposition.

\begin{proposition}
The following formulation is equivalent to MDPBL~\eqref{eq:MDPBL}:
\begin{equation}
\begin{array}{lll}
\text{MDPBL} \quad \displaystyle \min_{x,\eta} & \eta + \dfrac{1}{|T|} \lambda \sum_{t \in T} \sum_{s \in S} \sum_{a \in A} x_{tsa} \left( (c_{tsa}-\eta)^{+} \right) \\[1ex]
\text{s.t.} & x \in X\\
& \eta \text{ is free.}
\end{array}
\label{eq:MDPeta}
\end{equation}
\end{proposition}

\begin{proof}
A proof is provided in Subsection~\ref{Interpretation_of_CVaR}.
\end{proof}

Considering the optimization problem defined in \eqref{eq:MDPeta}, we now explain the lemma that further prunes the candidate set. Lemma \ref{lemma:threshold_eta} identifies the smallest threshold $\eta^{*}$ after which the objective becomes linear.

\begin{lemma}
\label{lemma:threshold_eta}
There exists a finite threshold \(\eta^*\), defined by the Markov structure, beyond which the optimal objective simplifies exactly to \(f^*(\eta) = \eta\). 
Specifically, we define:
\[
\eta^* = \min\{\eta : \exists x_{tsa}\geq 0\text{ feasible under Markov constraints, with } x_{tsa}=0\text{ if } c_{tsa}>\eta\}
\]
Then:
\[
f^*(\eta)=\eta,\quad\forall\eta\geq\eta^*.
\]
\end{lemma}

\begin{proof}
For the equality \(f^*(\eta)=\eta\) to hold at optimality, we must necessarily have:
\[
x_{tsa}=0,\quad\forall (t,s,a) \text{ with } c_{tsa}>\eta.
\]

While one might naively set $\eta = \max_{tsa} c_{tsa}$ so that $\{(t,s,a): c_{tsa}>\eta\}=\emptyset$, we observe the Markov constraints impose stronger structural restrictions: Probability mass at each timestep $t+1$ must precisely match probabilities transitioned from timestep $t$. This connectivity condition can make it impossible to restrict probabilities to lower-cost states/actions unless $\eta$ surpasses a threshold that ensures connectivity between these states/actions at all times.
This critical point arises from Markov constraints:
\begin{align*}
\sum_{a\in A} x_{(t+1)sa} = \sum_{k\in S}\sum_{a\in A} x_{tka}p_{tks}(a), \quad \forall t, \forall s.
\end{align*}

We define the candidate set of allowable states/actions for threshold \(\eta\) as:
\[
S_{\eta} = \{(t,s,a): c_{tsa}\leq\eta\}.
\]
We can call this $\eta$-safe set, as it is the set of state–action pairs whose cost does not exceed $\eta$. 
Then we define $\eta^*$:
\[
\eta^* = \min\{\eta : \exists x_{tsa}\geq0\text{ feasible under Markov constraints, with } x_{tsa}=0\text{ if } c_{tsa}>\eta\}.
\]
Clearly, for all \(\eta\geq\eta^*\), the set \(S_{\eta}\) includes the set \(S_{\eta^*}\), ensuring at least one stationary feasible solution exists, thus achieving:
\[
f^*(\eta)=\eta,\quad\forall\eta\geq\eta^*.
\]
Conversely, for \(\eta<\eta^*\), no feasible solution with support only in \(S_{\eta}\) exists (by minimality definition of \(\eta^*\)). Therefore, below \(\eta^*\), the objective \(f^*(\eta)\) must remain strictly larger than \(\eta\). Hence, it \(\eta^*\) precisely marks the threshold where the optimal objective value transitions from nonlinear to linear behavior in \(\eta\), completing the proof.
\end{proof}

This is illustrated in Figure \ref{fig:nonconvexity}(left), as at $\eta=1.8$ onward, $f^*(\eta)$ equals $\eta$.
This implies that the Markov condition restriction is already satisfied beyond this point, and further increasing \( \eta \) will not improve the objective function. 
Consequently, the minimum of \( f^*(\eta) \) must occur at or before the threshold \( f^*(\eta) = \eta \). This allows us to limit our candidate set of unique \( c_{tsa} \) values. Thus, we only need to test $\eta$ values from 
$min$ $c_{tsa}$ up to the first $c_{tsa}$ where \( f^{*}(\eta) = \eta \). This significantly reduces the computational burden. With this result, we introduce Algorithm \ref{al:algorithm1}. Using this algorithm, we can efficiently achieve the global optimum for MDPBL~\eqref{eq:MDPBL}.

\begin{algorithm}
    \caption{Finding the optimal \(\eta\)}
    \begin{algorithmic}[1]
        \State \(f^*,\eta^* \leftarrow \infty, \texttt{NaN}\)
        \ForAll{\(\eta \leftarrow \operatorname{sort}(\operatorname{unique}(c))\)}
        \State \(f \leftarrow f^*(\eta)\) \Comment{Solve \eqref{eq:subMDPBL}}
        \If{\(f < f^*\)}
        \State \(f^*, \eta^* \leftarrow f, \eta\) \Comment{Record new candidate}
        \EndIf
        \If {\(f = \eta\)}
        \State \textbf{break} \Comment{Entering linear growth regime}
        \EndIf
        \EndFor
        \State \Return \(\eta^*\)  \Comment{Global minimizer}
    \end{algorithmic}
    \label{al:algorithm1}
\end{algorithm}

\section{Numerical Results}
\label{sec:results}
\subsection{Thermal backup of offshore wind}

To assess the impact of incorporating CVaR into our base model, we compare the resulting risk-averse policies against the baseline risk-neutral plan for the thermal backup example. Figure \ref{fig:method1flexibleplan} shows the differences between the risk-neutral operational plan and those generated with our model with different levels of risk aversion. 
Increasing $\beta$ means more generation in order to prevent high-cost, low-probability outcomes.
Table \ref{tab:all_variations_concise_updated} presents the total curtailed demand and total system costs from simulations across 17 years for the thermal backup of offshore wind application. It demonstrates how increased risk aversion ($\beta$) reduces demand curtailments at the cost of higher operational costs. As expected, larger values of $\beta$ prevent worst-case outcomes, leading to more conservative actions. In this case, the total curtailed demand decreases while the total system costs increase.

\begin{figure}[H]
\centering
\begin{minipage}{.5\textwidth}
\centering
  \includegraphics[width=1\textwidth]{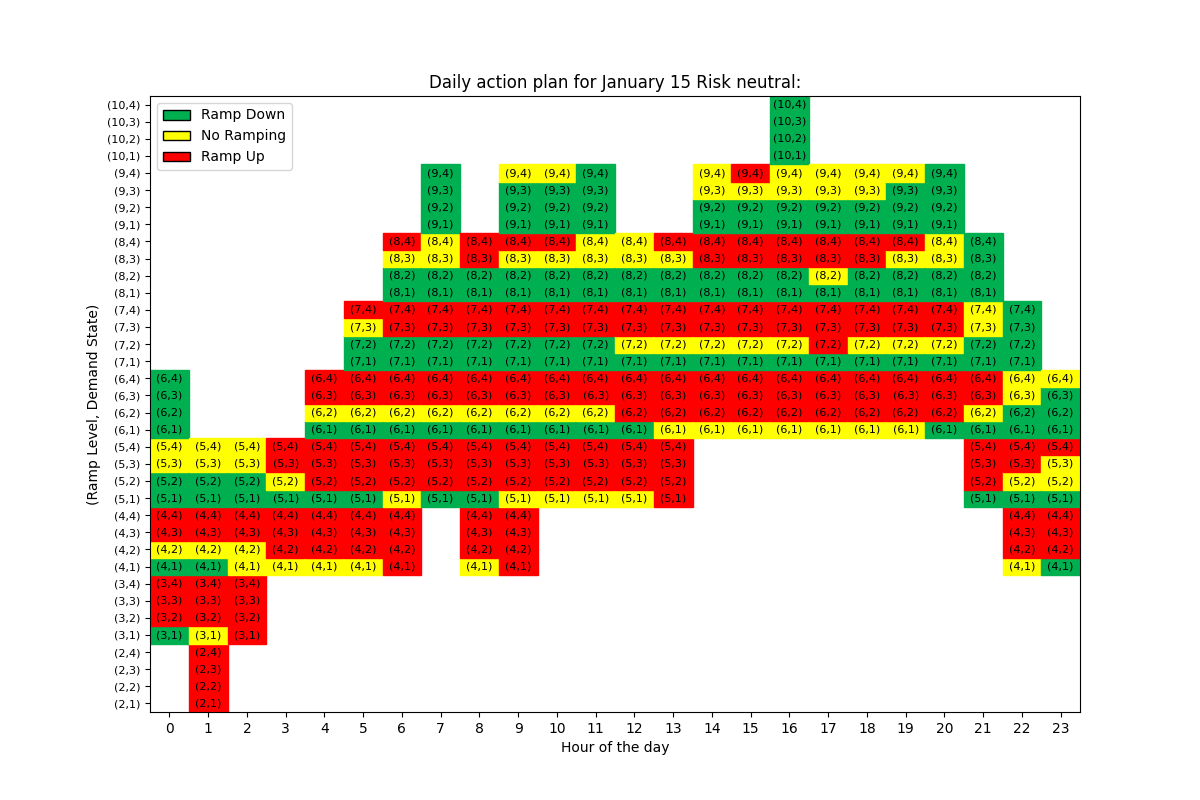}
\end{minipage}%
\begin{minipage}{.5\textwidth}
\centering
  \includegraphics[width=1\textwidth]{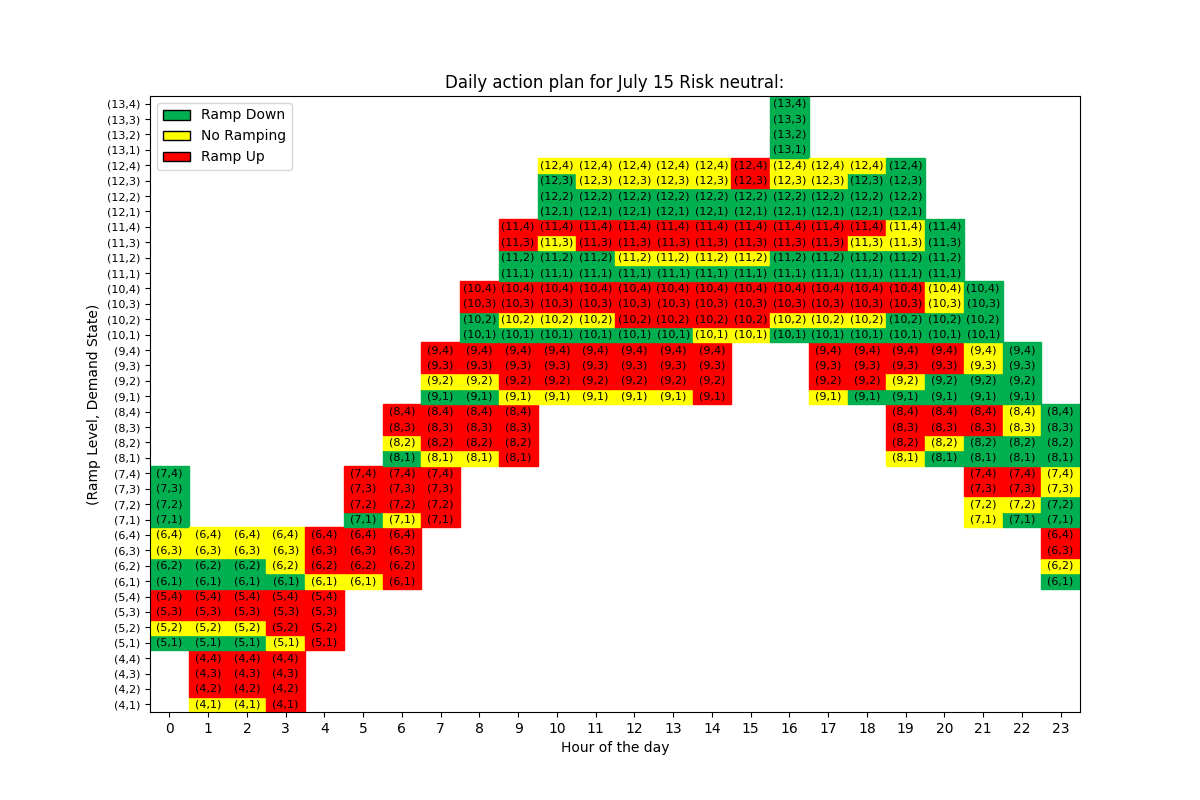}
\end{minipage}


\begin{minipage}{.5\textwidth}
\centering
  \includegraphics[width=1\textwidth]{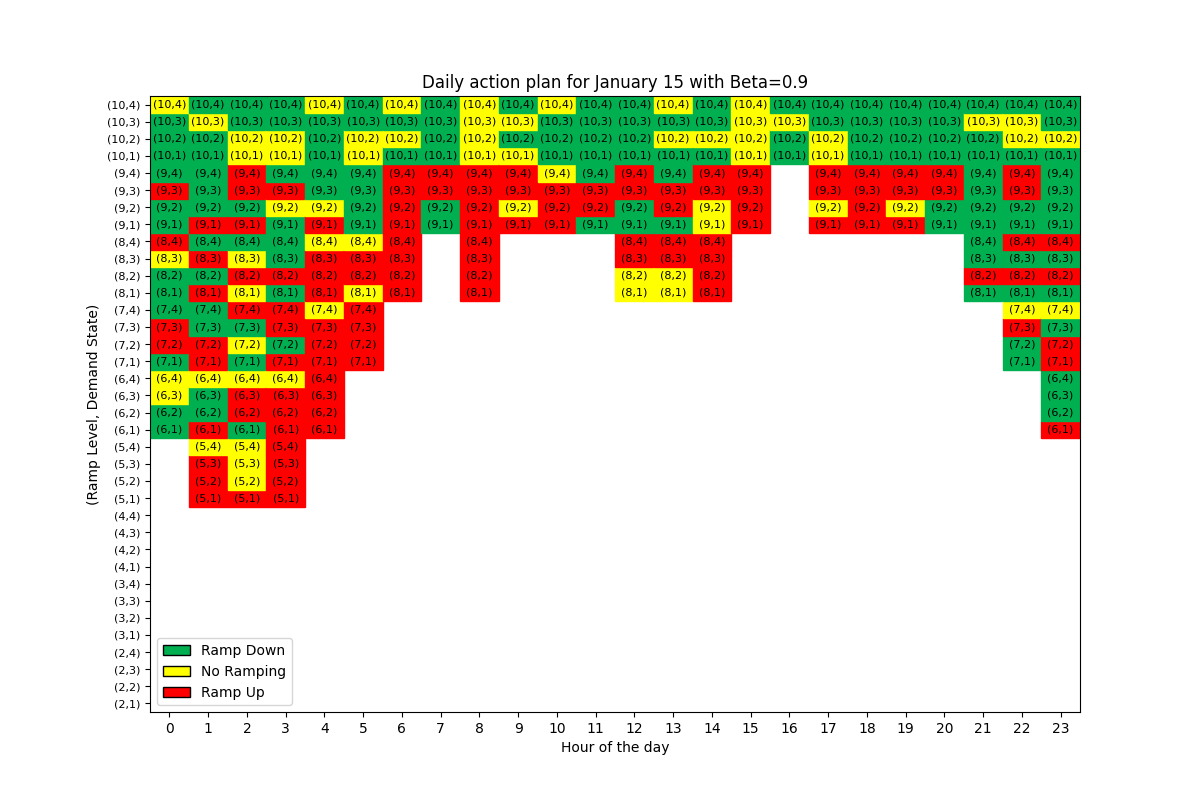}
\end{minipage}%
\begin{minipage}{.5\textwidth}
\centering
  \includegraphics[width=1\textwidth]{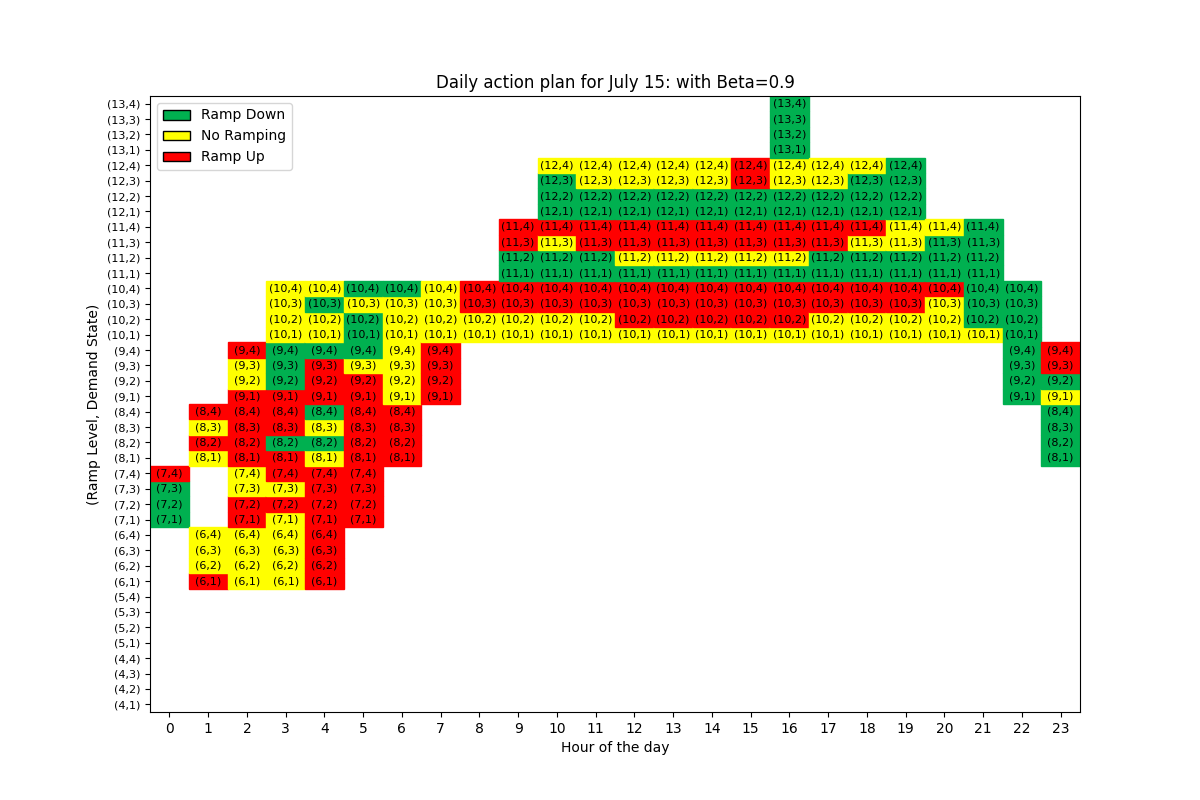}
\end{minipage}

\caption{Thermal generation action plans (hourly ramp decisions) for representative winter and summer days. Columns show Jan. 15 (left) and Jul. 15 (right); rows show the risk-neutral policy (top) and the risk-averse policy with $\beta=0.9$ (bottom). Colors indicate the selected action: ramp down (green), no ramping (yellow), and ramp up (red).}
\label{fig:method1flexibleplan}
\end{figure}

\subsection{Hydropower reservoir management}

Figures \ref{fig:hydro-policy-0.0}--\ref{fig:hydro-policy-0.9} depict the optimal operating policies for the risk-neutral version alongside with \(\beta \in \{0.5, 0.9\}\) in terms of the respective amount of thermal generation in each state---which is the complement of hydropower generation in the assumed fixed load to be served.
Time \(t\) varies on the horizontal axis; regime \(r\) varies along the plotting panes, and reservoir level \(\ell\) along the vertical axis of each pane.

\afterpage{%
  \clearpage
  \thispagestyle{empty}
    \begin{figure}[p]
        \centering
        \begin{subfigure}[b]{0.48\textwidth}
            \includegraphics[width=\textwidth]{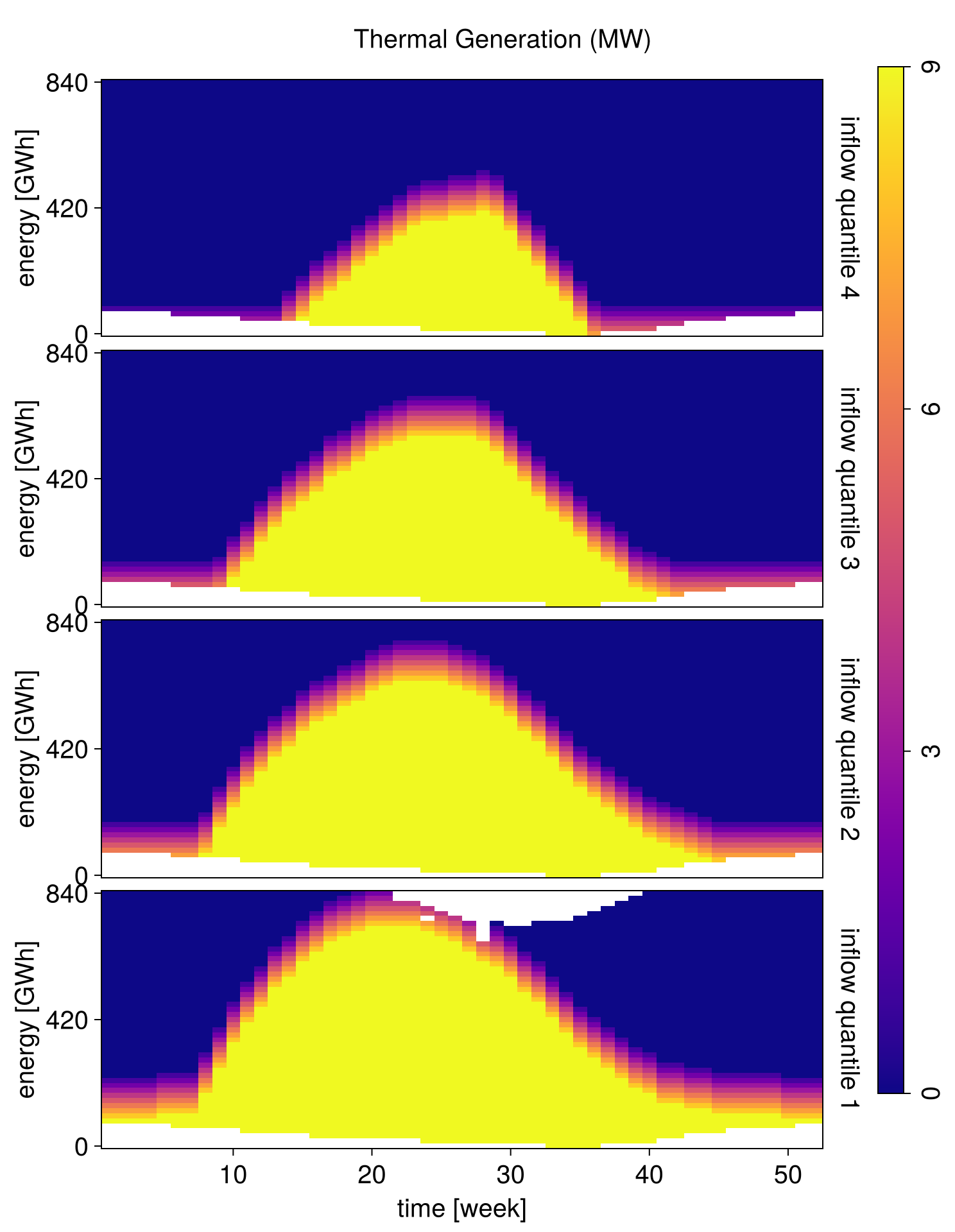}
            \caption{Risk neutral}
            \label{fig:hydro-policy-0.0}
        \end{subfigure}
    
        \centering
        \begin{subfigure}[b]{0.48\textwidth}
            \includegraphics[width=\textwidth]{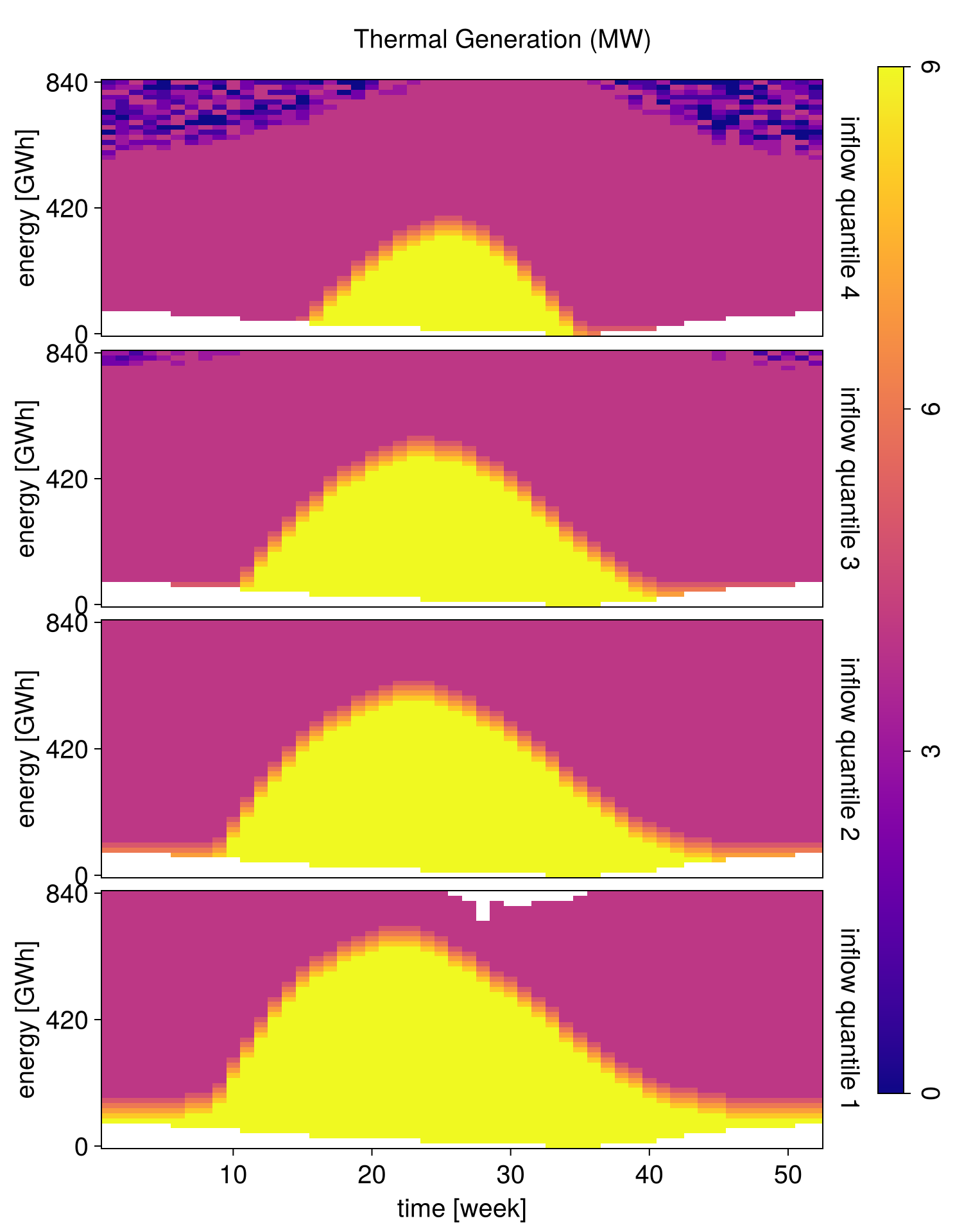}
            \caption{$\beta=0.5$}
            \label{fig:hydro-policy-0.5}
        \end{subfigure}
        \medskip
        \begin{subfigure}[b]{0.48\textwidth}
            \includegraphics[width=\textwidth]{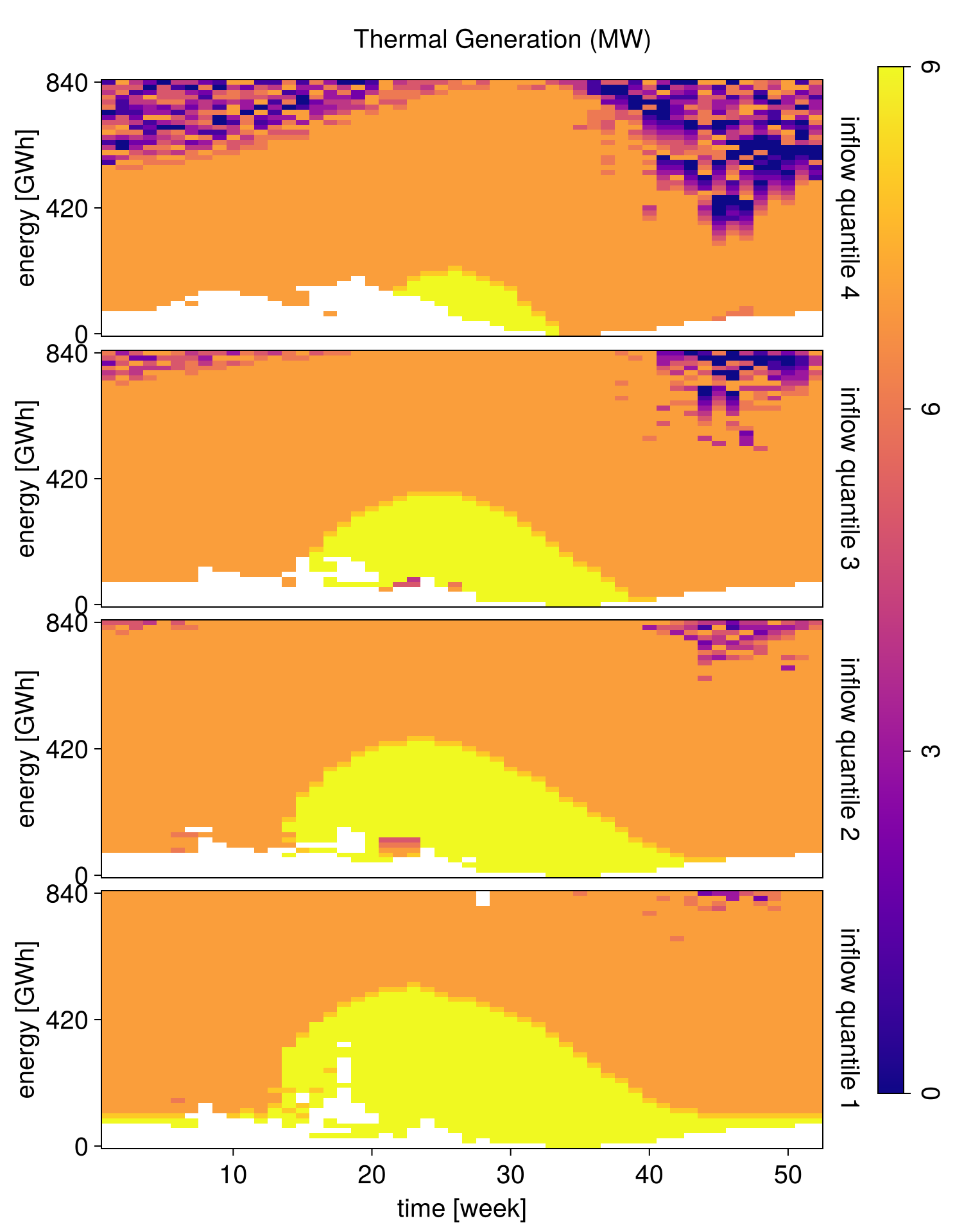}
            \caption{$\beta=0.9$}
            \label{fig:hydro-policy-0.9}
        \end{subfigure}
        \caption{Optimal reservoir operating policies shown as thermal generation (MW) needed to meet the fixed load (i.e., the complement of hydropower output) over time $t$ and storage level $\ell$, for each inflow regime $r\in\{1,\dots,4\}$. Panels: (a) risk-neutral, (b) $\beta=0.5$, (c) $\beta=0.9$.}
    \end{figure}
  \clearpage
}
For the Waitaki Catchment, the largest seasonal inflow has typically been observed in the summer (December to February), with lowest inflow during the winter (June to August) \citep{Pritchard_hydro, Zheng2007}.

The uncolored (white) areas visible at the lower edge of each pane indicate states that are \emph{not} visited when following the respective operating policy.
For each risk level \(\beta\), several trends are evident in the plots:
\begin{itemize}
    \item Decreasing thermal generation with increasing \(\ell\) (wetter regime).

    \item Longer periods of thermal generation at lower lake levels during the (dry) middle months: The yellow region at the base of each pane is centered approximately on the dry (winter) months and decreases in breadth and height with increasing inflow regime \(r\).

    \item Increasing thermal (decreasing hydropower) generation at higher storage levels with increasing risk aversion \(\beta\).

\end{itemize}

Two other features are, perhaps, less intuitive:
\begin{itemize}
    \item The transition between \textit{some} hydropower generation and \textit{none} (following any vertical section) is increasingly sharp with increasing \(\beta\).

    \item Both of the risk-averse policies exhibit apparent ``ambivalence'' at high lake levels in the high-inflow regime \(r = 4\) over the wet months (November to February, say), manifest in the ``patchwork'' pattern wherein a small change in state \((t, \ell)\) can produce a large change in hydro-thermal split. This suggests increasing numerical condition number in the linear programs with respect to increasing \(\beta\).

\end{itemize}

\section{Chance constraints with linear constraints}
\label{sec:LPvariations}
This section explores several approaches that grid operators can employ to enhance grid reliability. We will only focus on the thermal backup of the offshore wind model, and the examples and results are only for this application.
Each of these approaches can be adopted as an alternative to the CVaR-based method in Section \ref{sec:riskaverseMDP}. All these variations share similar underlying concepts, as each introduces additional linear constraints to our primary formulations MDPLP~\eqref{eq:MDPLP} to address specific reliability concerns. These variations maintain a linear program structure and therefore remain straightforward to solve. 
In some of the variations, the state space must be expanded to incorporate additional information, but the overall MDP structure remains unchanged. 
Before presenting these variations in detail, we note that the optimal decisions resulting from adding these to our main problems will still remain deterministic, as highlighted in Proposition \ref{prop:deterministic3}.

\begin{proposition}
    \label{prop:deterministic3}
    When \textsc{MDPLP}~\eqref{eq:MDPLP} is augmented with any of the linear constraints introduced in this section, there exists an optimal policy obtained by solving them that remains deterministic rather than randomized.
\end{proposition}

\subsection{Variation 1: A limit on the average amount of curtailed demand}
We can impose a limit on the average amount of electricity demand curtailment over a specific time horizon, such as individual days or the entire year.
This approach is particularly useful when grid operators aim to manage excess demand through demand response programs. Given that the available capacity of these programs is known, operators can strategically constrain average curtailment to ensure that only a manageable portion will be left unmet.

\begin{itemize}
    \item \textbf{Daily constraint:} Limit on the average amount of demand curtailment in each day $n$:
          \begin{align}
               & \sum_{t = (24n - 23)}^{24n} \sum_{s \in S} \sum_{a \in A} \text{CD}_{ts} \cdot x_{tsa} \leq \tau_{day} \nonumber \text{ for } n = 1, \ldots ,365; \\
          \end{align}

    \item \textbf{Annual constraint:} Limit the average demand curtailment across the entire year:
          \begin{align}
               & \sum_{t = 1}^{|T|=8760} \sum_{s \in S} \sum_{a \in A} \text{CD}_{ts} \cdot x_{tsa} \leq \tau_{year}; \nonumber \\
          \end{align}

\end{itemize}

Figure \ref{fig:timedependentwithextra} shows the action plans resulting from this variation using annual constraint with $\tau=43.8$ \,GW. As shown, imposing a constraint on average curtailed demand leads to more conservative operational decisions.
Table \ref{tab:all_variations_concise_updated} presents a detailed comparison of simulation results, highlighting the effects of this variation on both total system costs and curtailed demand.

\begin{figure}[t]
\centering
\begin{minipage}{.5\textwidth}
    \centering
    \includegraphics[width=1\textwidth]{Final_Pictures/Risk-neutral-January-15.png}
\end{minipage}%
\begin{minipage}{.5\textwidth}
    \centering
    \includegraphics[width=1\textwidth]{Final_Pictures/Risk-neutral-July-15.png}
\end{minipage}
\begin{minipage}{.5\textwidth}
    \centering
    \includegraphics[width=1\textwidth]{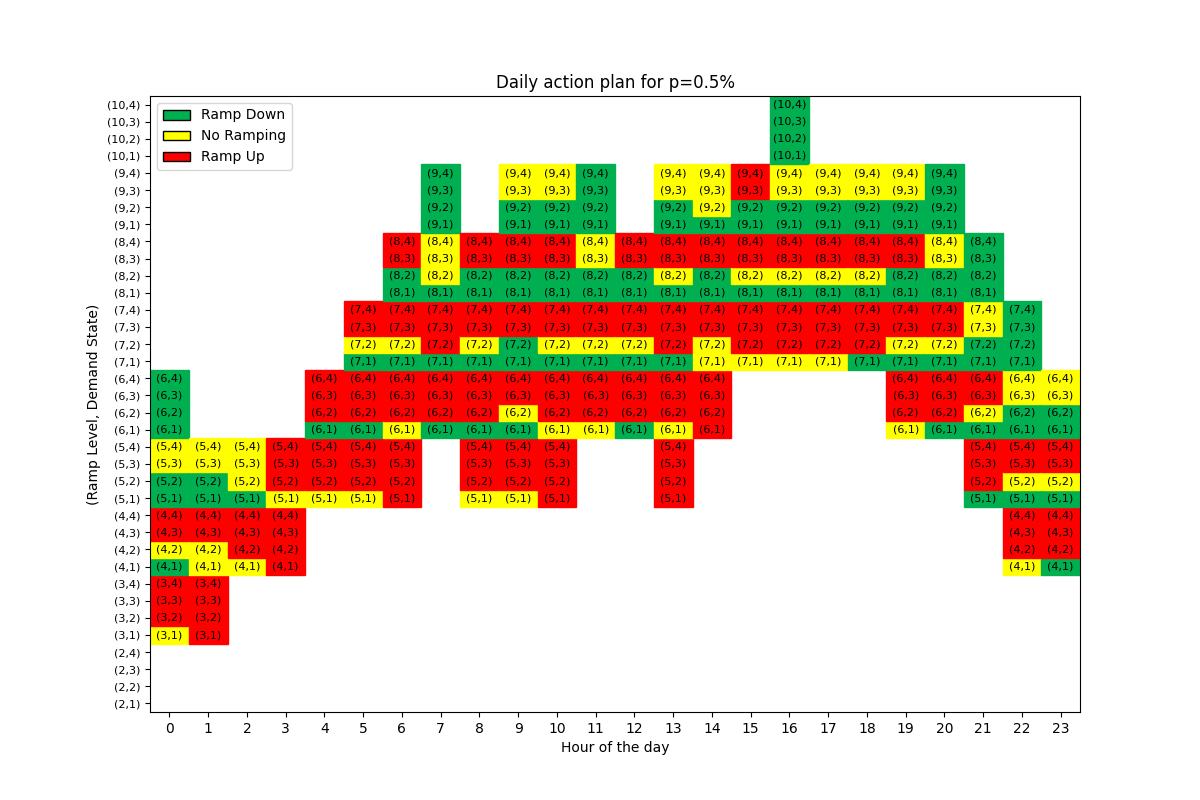}
\end{minipage}%
\begin{minipage}{.5\textwidth}
    \centering
    \includegraphics[width=1\textwidth]{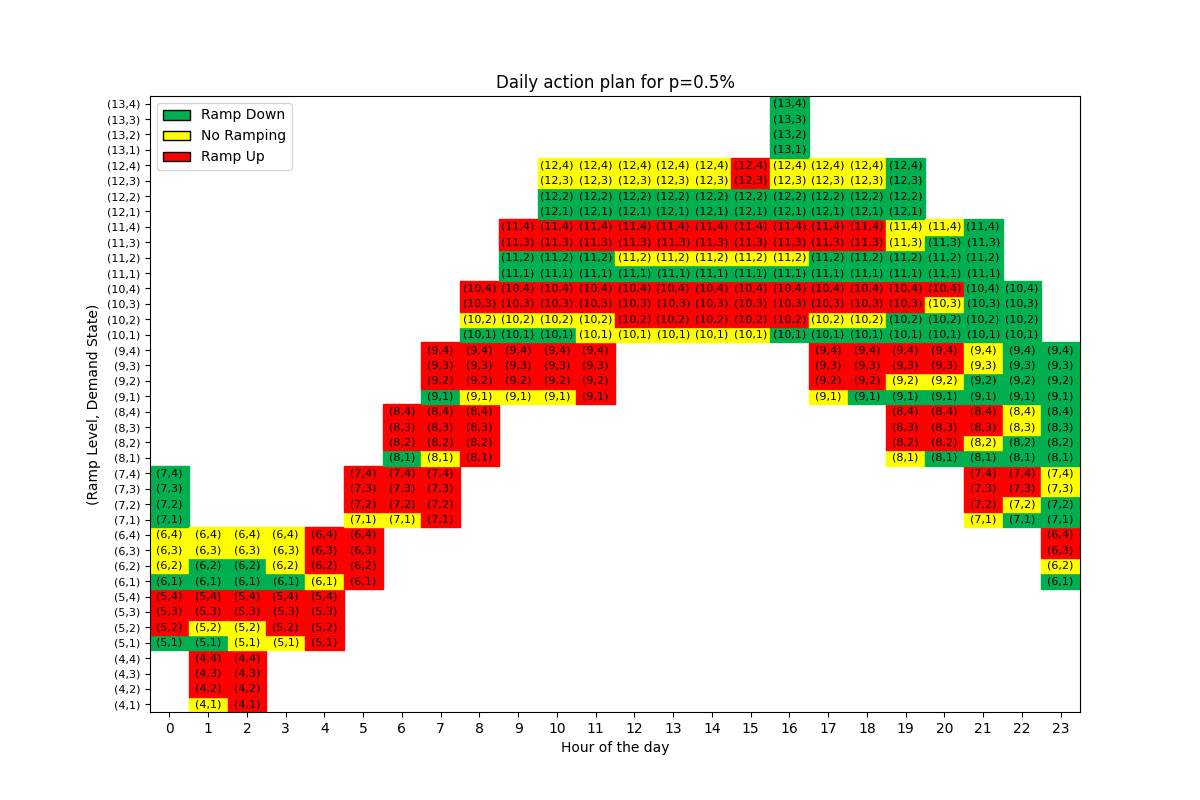}
\end{minipage}

\caption{Effect of an annual cap on average demand curtailment (Variation 1). Daily action plans for Jan. 15 (left) and Jul. 15 (right). The top row shows the baseline (unconstrained) risk-neutral policy, while the bottom row enforces the annual curtailment constraint with $\tau_{\text{year}}=\SI{43.8}{\giga\watt}$, limiting the average amount of unmet demand over the year.}
\label{fig:timedependentwithextra}
\end{figure}


\subsection{Variation 2: A limit on the probability of being in a curtailed state}

Another risk aversion approach involves limiting the probability of the system entering a state associated with demand curtailment. This probability constraint can be applied for specific days or the whole year. 
To implement this, first we define a binary parameter $e_{t}$ that indicates whether curtailment occurs at time $t$.

\begin{itemize}
    \item $e_{t} =
              \begin{cases}
                  1 & \text{if } \text{CD}_{ts} > 0, \\
                  0 & \text{otherwise}
              \end{cases}
              \quad \forall t \in T, \quad \forall (\ell, r) \in S $
    \item $s=(\ell,r,e_{t})$ state of MDP
\end{itemize}

As $e_{t}$ is a parameter and not a decision variable, incorporating it into the formulation does not change the linear nature of the problem. 
Depending on the desired time frame, the following constraints can be used to limit the frequency of curtailment.

\begin{itemize}
    \item \textbf{Daily constraint:} Limit on the average daily probability of demand curtailment on each day $n$:

          \begin{align}
               & \frac{1}{24} \sum_{t = (24 \cdot n - 23)}^{24 \cdot n} \sum_{s = (\ell,r,e=1) \in S} \sum_{a \in A} x_{tsa} \leq \mathcal{P}_{day} \nonumber \text{ for } n = 1, \ldots ,365; \\
          \end{align}

    \item \textbf{Annual constraint:} Limit on the average probability of demand curtailment across the year:
          \begin{align}
               & \frac{1}{|T|} \sum_{t = 1}^{|T|=8760} \sum_{s = (\ell,r,e=1) \in S} \sum_{a \in A}  x_{tsa} \leq \mathcal{P}_{year}; \nonumber \\
              \label{eq:Var2}
          \end{align}
\end{itemize}

Figure \ref{fig:timedependentwithy} compares this variation with the original plans for the time-dependent plan.
Table \ref{tab:all_variations_concise_updated} 
presents a detailed comparison of simulation results, highlighting the effects of the curtailment constraint on both system costs and curtailed demand.

\begin{figure}[t]
    \centering
    \begin{minipage}{.5\textwidth}
        \centering
        \includegraphics[width=1\textwidth]{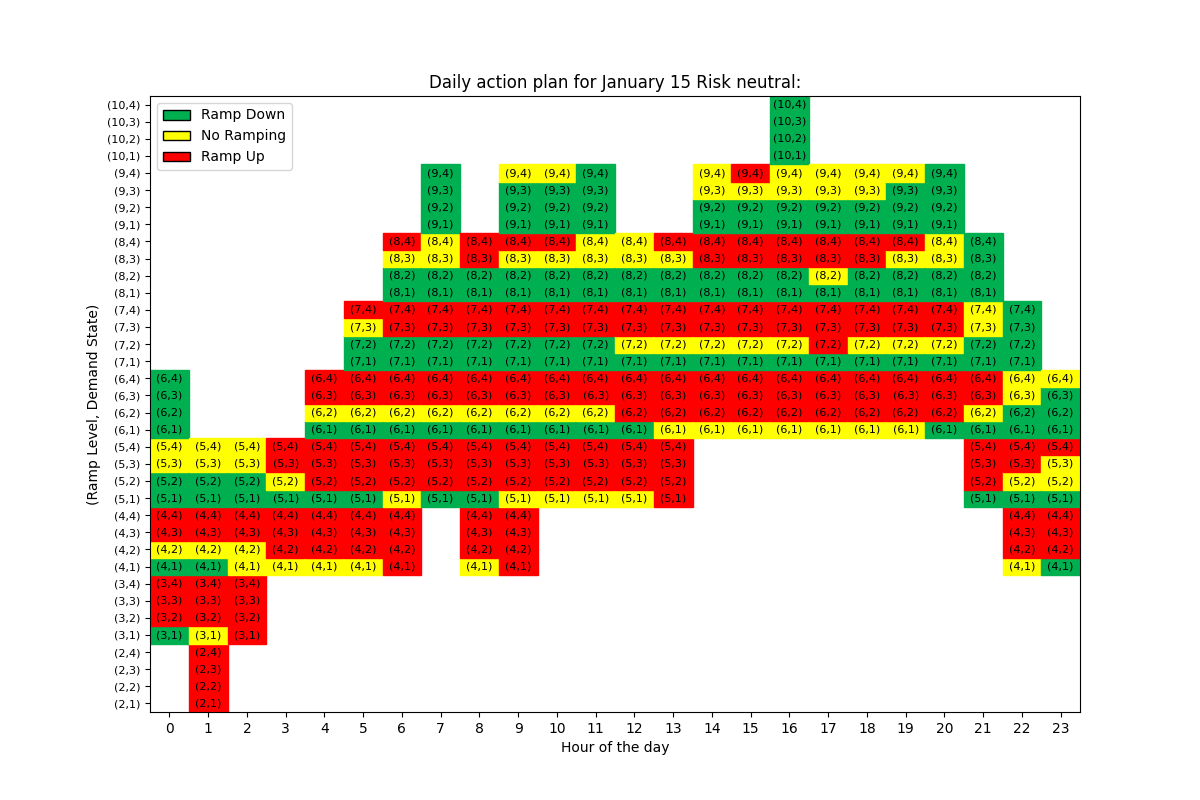}
    \end{minipage}%
    \begin{minipage}{.5\textwidth}
        \centering
        \includegraphics[width=1\textwidth]{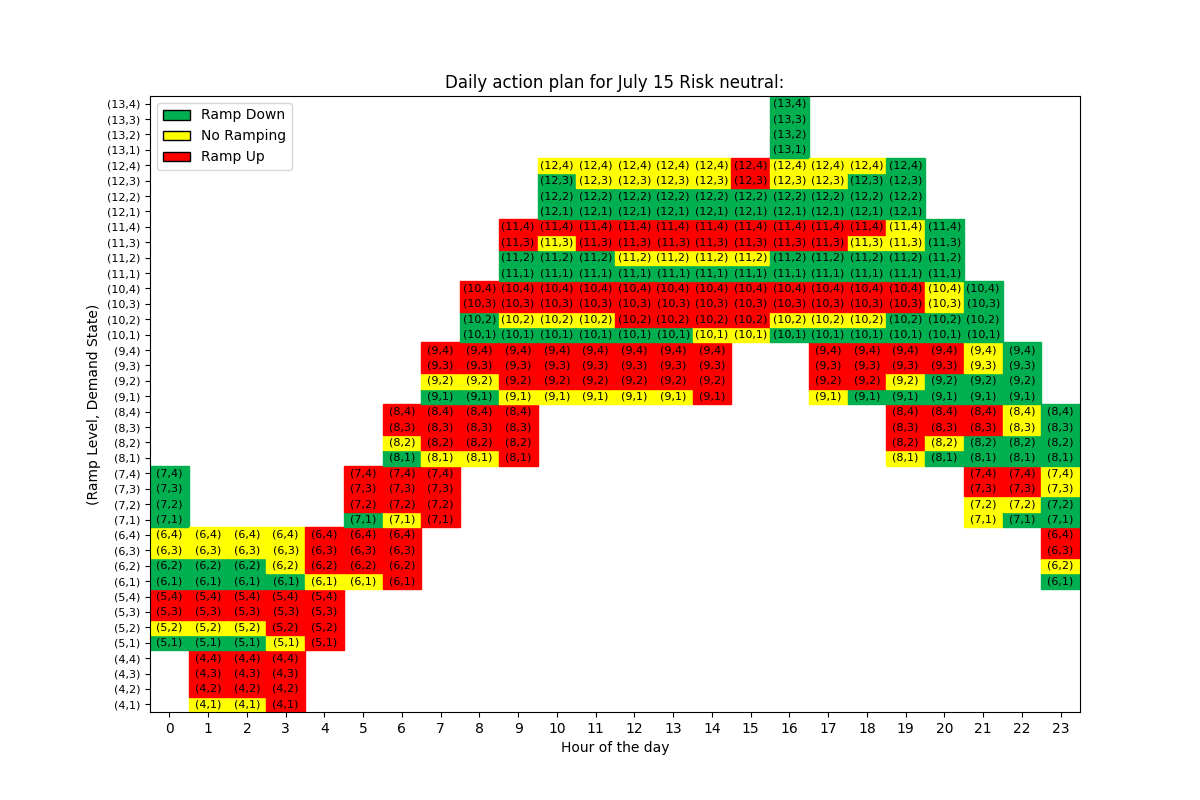}
    \end{minipage}

    \begin{minipage}{.5\textwidth}
        \centering
        \includegraphics[width=1\textwidth]{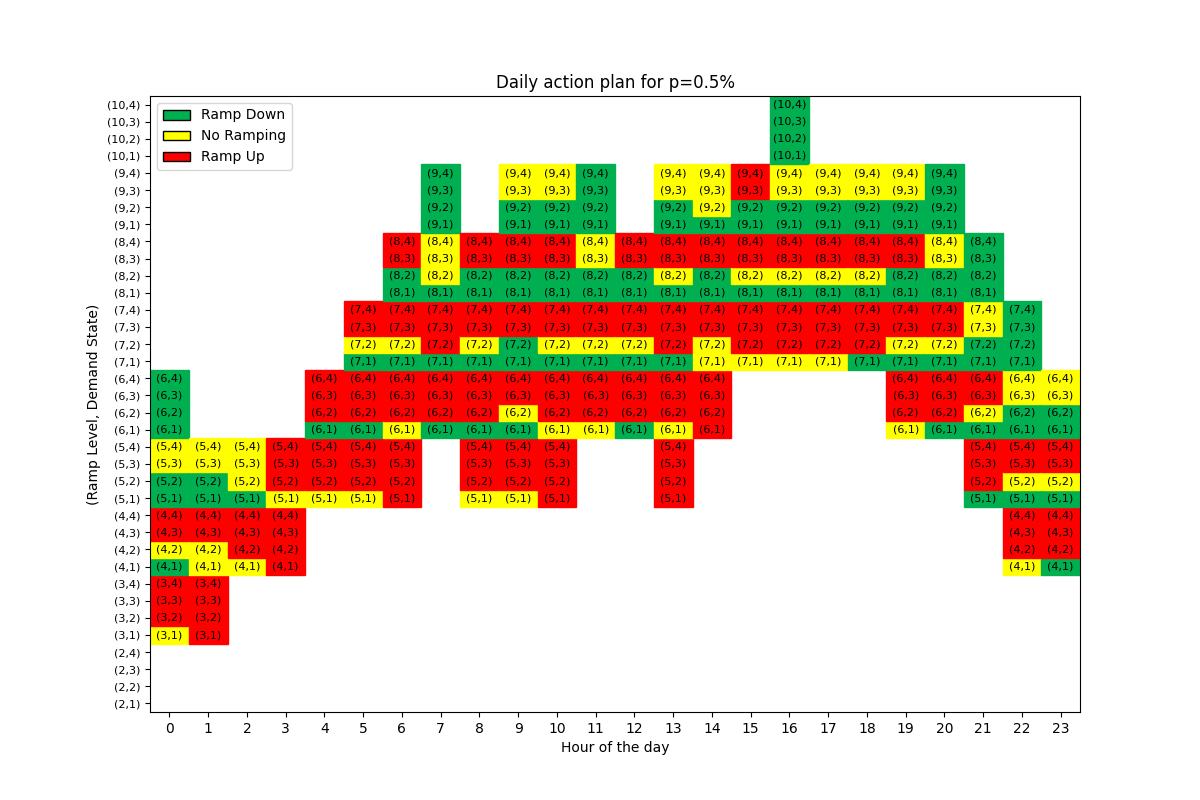}
    \end{minipage}%
    \begin{minipage}{.5\textwidth}
        \centering
        \includegraphics[width=1\textwidth]{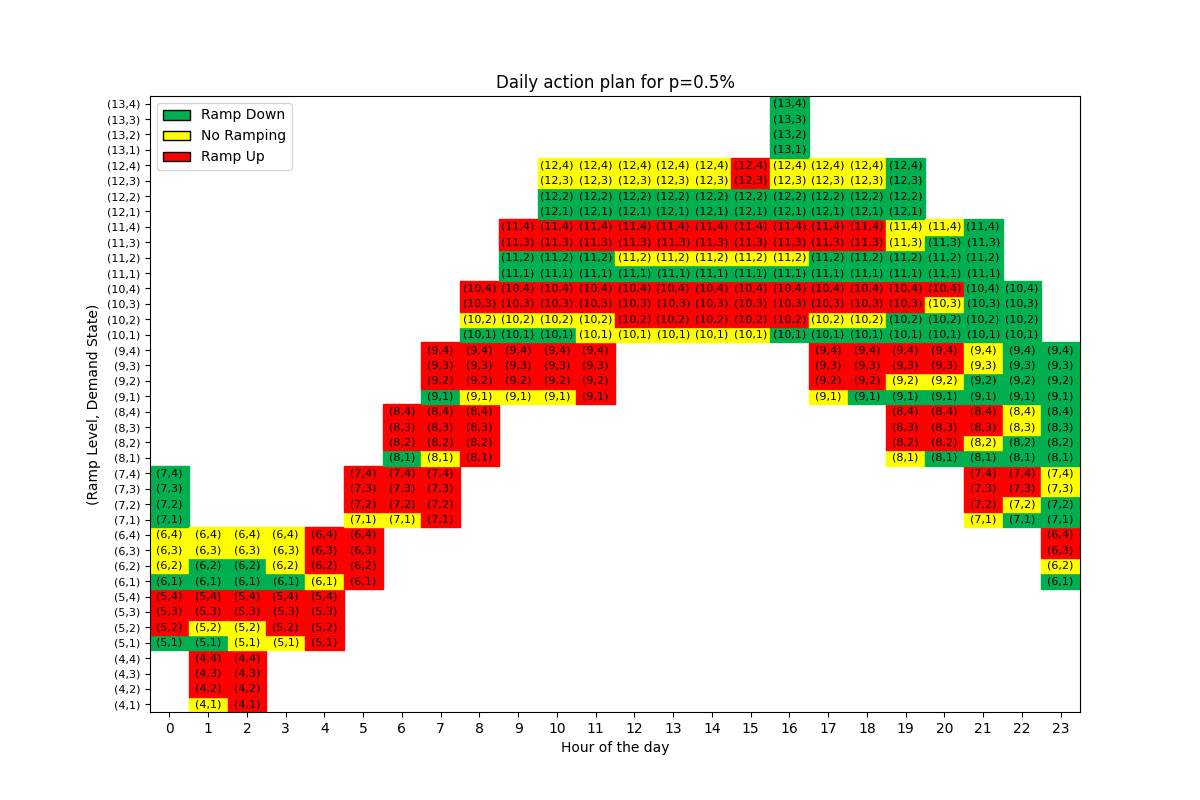}
    \end{minipage}

    \caption{Effect of an annual cap on the probability of curtailment (Variation 2). Daily action plans for Jan. 15 (left) and Jul. 15 (right). The top row shows the baseline (unconstrained) risk-neutral policy, while the bottom row enforces the annual curtailment-state probability constraint (17), limiting the fraction of hours in which curtailment occurs to ($\mathcal{P}_{\text{year}}=0.5\%)$}
    \label{fig:timedependentwithy}
\end{figure}


\subsection{Variation 3: A limit on the probability of consecutive curtailment of demand}
In a real-world grid, when a curtailment in electricity demand occurs for just one hour, it may be considered a minor issue, even if it leads to a temporary power outage.

While undesirable, individuals can continue their activities, relying on batteries for essential tasks. However, the situation becomes significantly more severe when power demand curtailments last more than a single hour, especially when they occur consecutively. It is easier for grid operators to address load curtailment for one hour with a demand response program, but consecutive curtailments greatly increase the risk of prolonged power outages.

Additionally, when they are consecutive, they can disrupt people's daily routines, from heating and cooling their homes to preserving their food in refrigerators, and are especially detrimental to vulnerable populations. This situation can lead to food spoilage and unsafe living conditions. Hospitals and other facilities must also use backup systems like generators to keep the essential systems running, risking people's lives as outages last longer than one consecutive hour. Lastly, the continuity of such outages increases psychological stress among individuals and communities.

To address these challenges, we redesigned our original MDP states to provide an approach restricting consecutive electricity demand curtailments.\\
\\
Let $z$ be the number of consecutive curtailment periods until now, so that

\begin{itemize}
    \item $z_t =
              \begin{cases}
                  min(z_{t-1} + 1,n) & \text{if curtailment occurs at time } t \\
                  0                  & \text{otherwise}
              \end{cases}$
\end{itemize}

\begin{itemize}
    \item $n =$ the number of consecutive curtailments of demand to restrict
    \item $s=(\ell,r,z_t)$ state of MDP
\end{itemize}

\begin{itemize}
    \item \textbf{Annual constraint:} Limit on $n$ or more consecutive curtailment of demand across the year:
          \begin{align}
               & \frac{1}{|T|} \sum_{t \in T} \sum_{(\ell,r,z_{t} \geq n)\in S} \sum_{a \in A} x_{tsa} \leq \mathcal{P}_{year} \nonumber \\
          \end{align}
          \label{eq:Var3}
\end{itemize}

For example, we can prevent two or more, or three or more consecutive hours of electricity demand curtailment by setting $\mathcal{P}=0$ in the corresponding constraint. Figure~\ref{fig:ConsecutiveCurtailmentsTimeDependent} presents action plans under four scenarios: one with no demand curtailment allowed, two with restrictions on two and three consecutive curtailments, respectively, and one baseline scenario with no curtailment constraints. It is important to note that the scenario with no curtailment permitted throughout the entire year was infeasible under the existing setup of 13 CCGTs, achieving feasibility required one additional CCGT. Table \ref{tab:all_variations_concise_updated} compares the total curtailed demand and the total costs associated with the system under each scenario.

\subsection{Variation 4: A penalty for consecutive curtailment of demand in the cost function}

Instead of enforcing a hard constraint on consecutive curtailments of electricity demand, an alternative approach is to incorporate penalties for consecutive curtailments that occur directly into the cost function. This penalty should be convex on the number of consecutive curtailment periods. Unlike the previous variations, which introduced additional constraints to the primary framework, this variation changes the immediate costs attached to each state. Nevertheless, it is relevant to include in this section, as it still serves our goal of restricting consecutive curtailment of demand.

This is achieved by introducing a penalty $c_3(z)$, which is added to the base cost of each state and increases with the number of consecutive curtailments $z$. This provides a flexible trade-off. Under this formulation, the system may occasionally permit consecutive curtailments but discourages them by incurring additional penalties. The resulting LP reflecting this approach will be as follows.

\[
    \text{new } c_{tsa} = \text{old } c_{tsa} + c_{3}(z_t)
\]

Additionally, it is possible to combine both of these approaches, having a penalty as well as a hard constraint on curtailment.

\clearpage
\thispagestyle{empty}
\begin{figure}[p]
        \centering
        \begin{minipage}{.49\textwidth}
            \centering
            \includegraphics[width=1\textwidth]{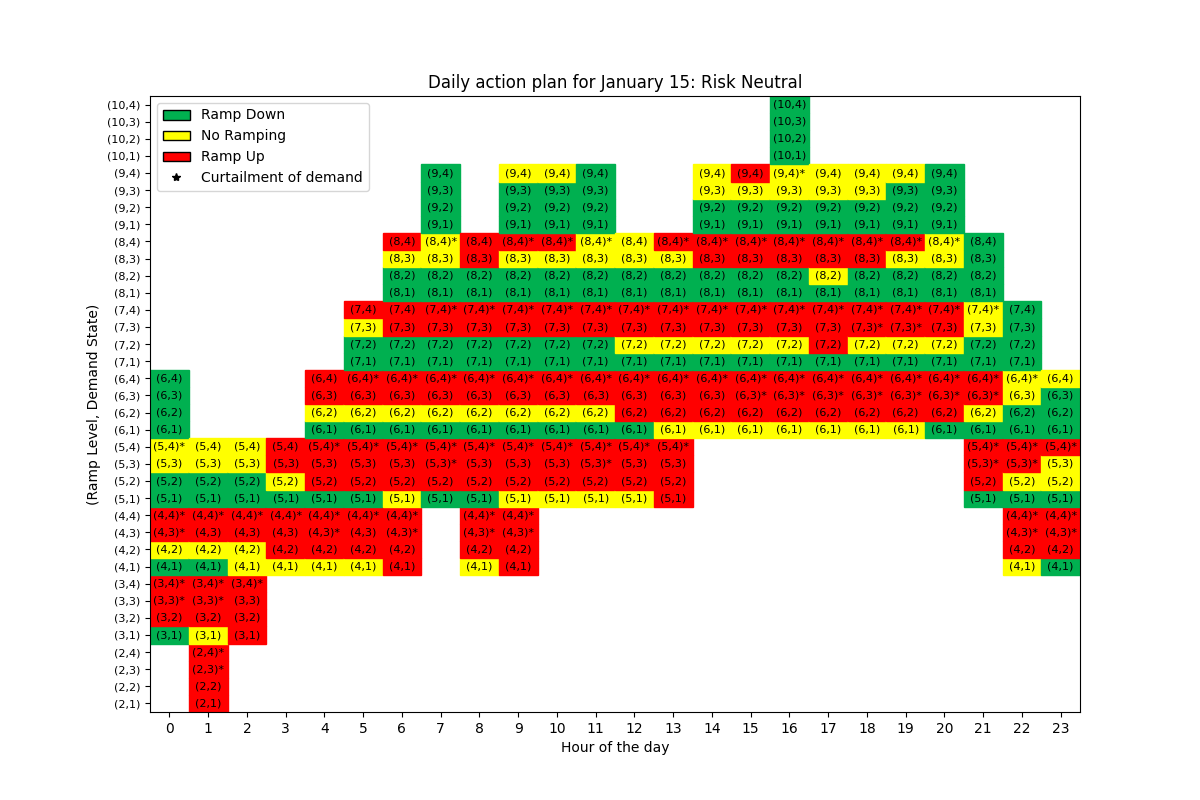}
        \end{minipage}
        \begin{minipage}{.49\textwidth}
            \centering
            \includegraphics[width=1\textwidth]{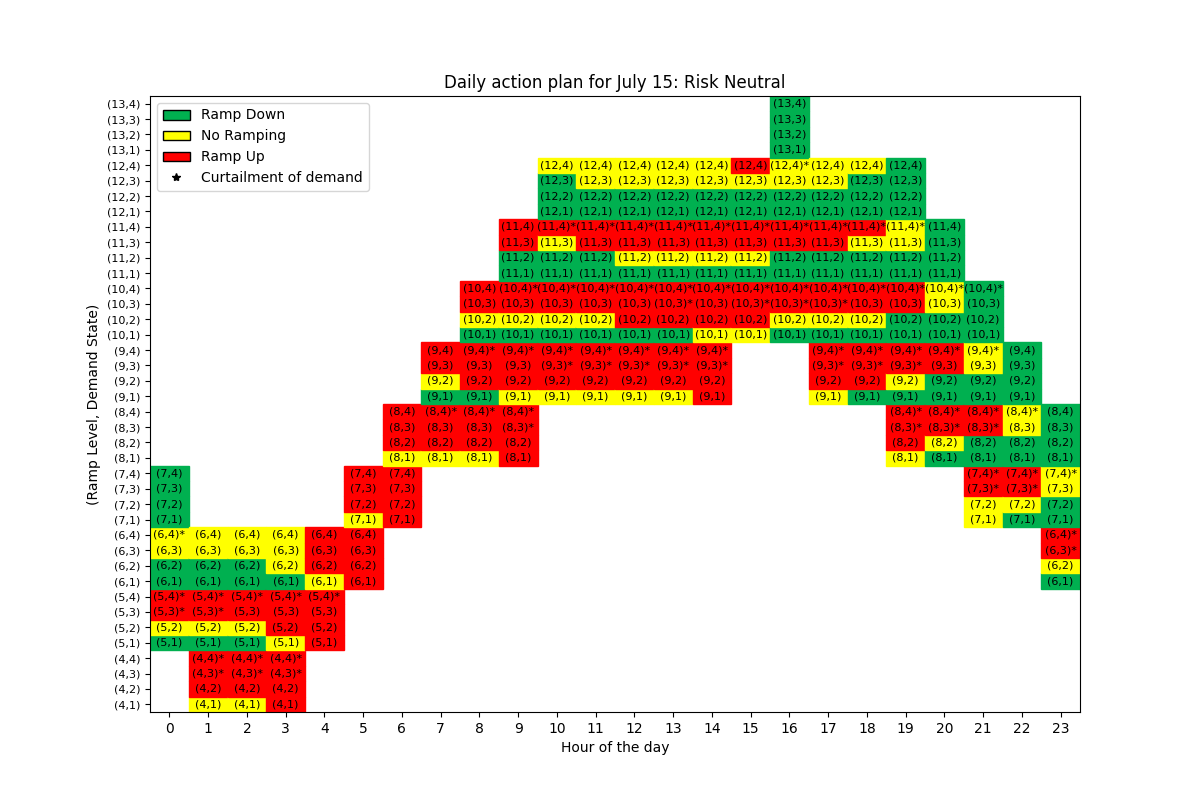}
        \end{minipage}

        \begin{minipage}{.49\textwidth}
            \centering
            \includegraphics[width=1\textwidth]{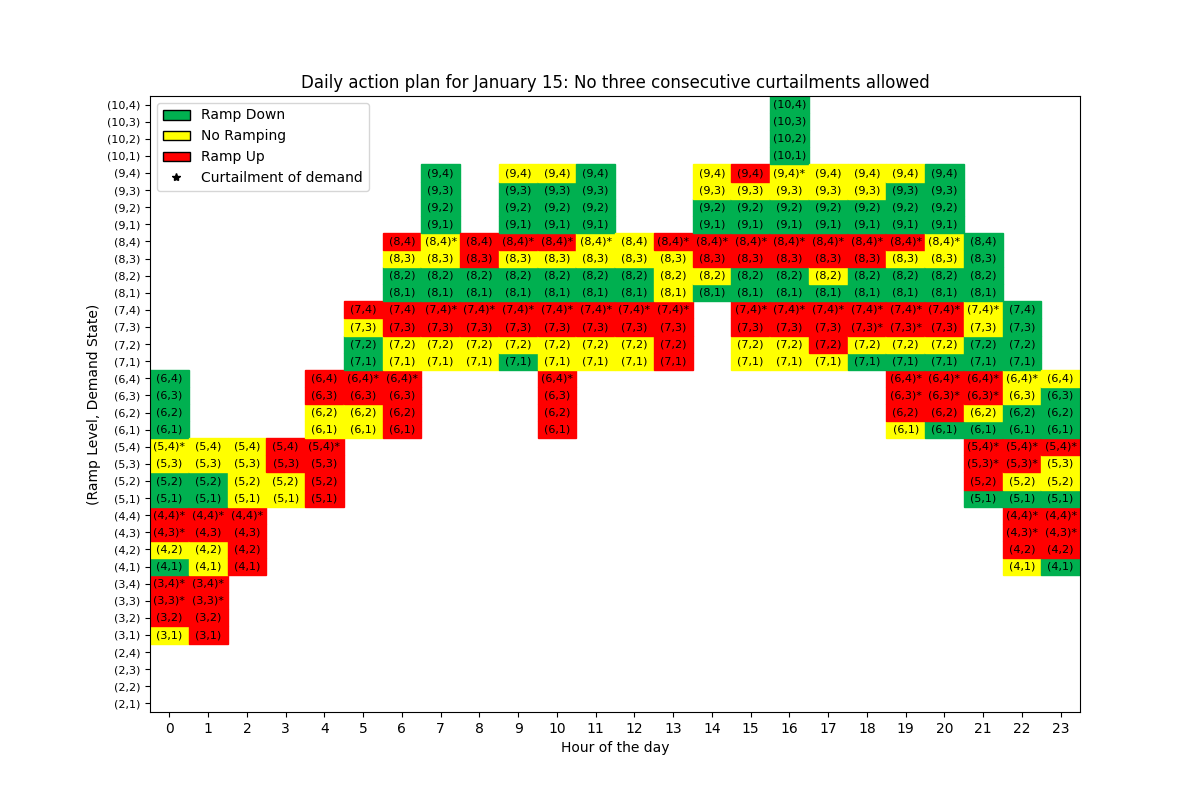}
        \end{minipage}%
        \begin{minipage}{.49\textwidth}
            \centering
            \includegraphics[width=1\textwidth]{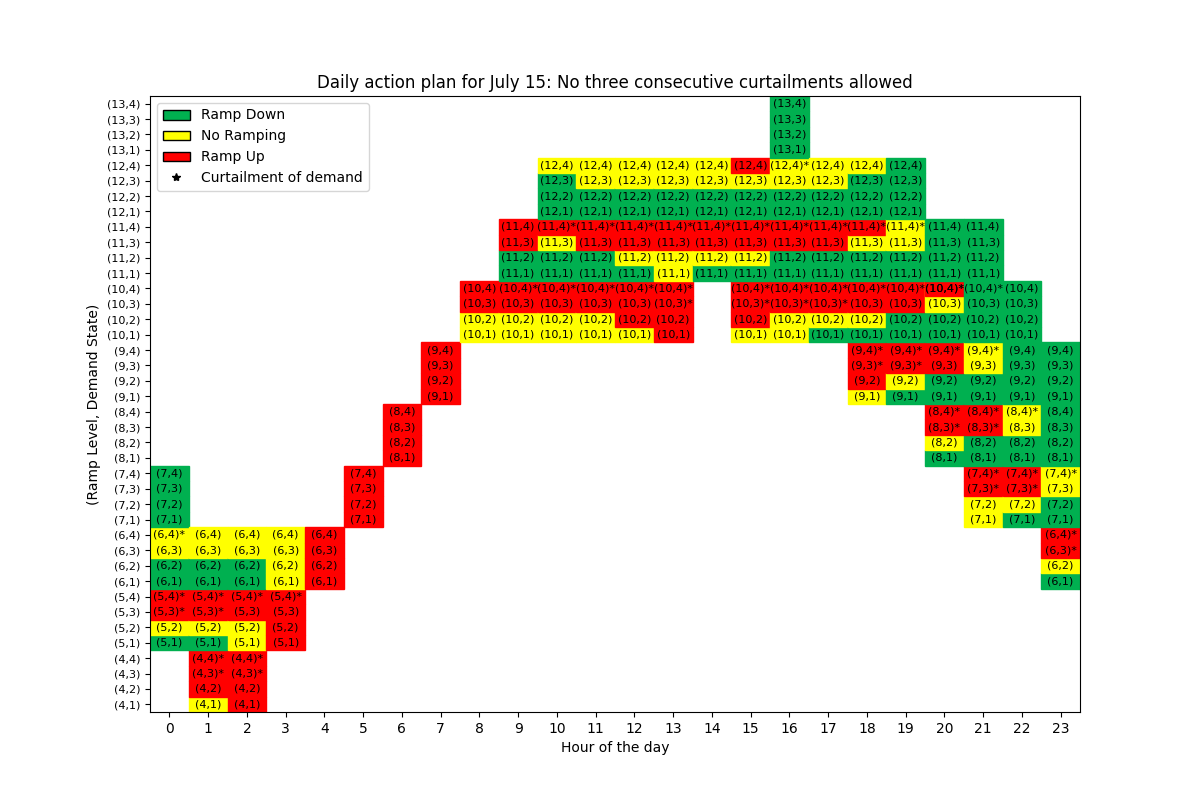}
        \end{minipage}
    
        \begin{minipage}{.49\textwidth}
            \centering
            \includegraphics[width=1\textwidth]{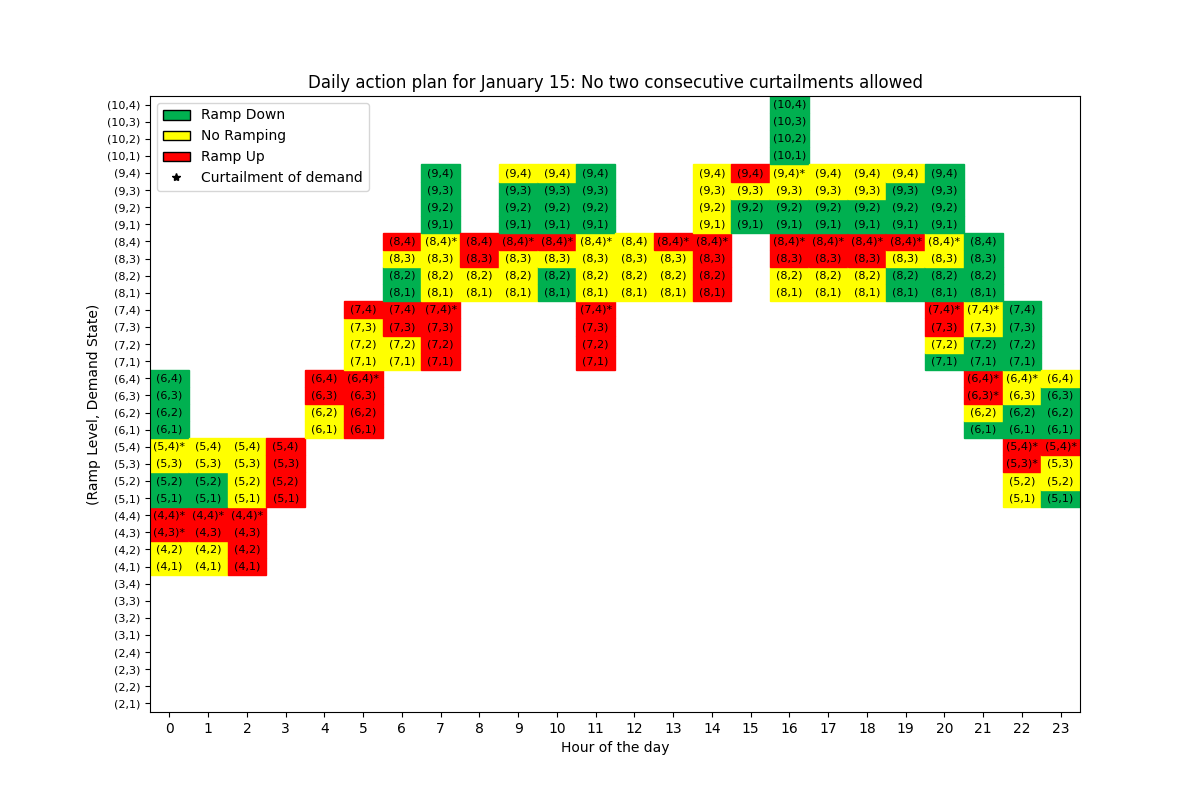}
        \end{minipage}%
        \begin{minipage}{.49\textwidth}
            \centering
            \includegraphics[width=1\textwidth]{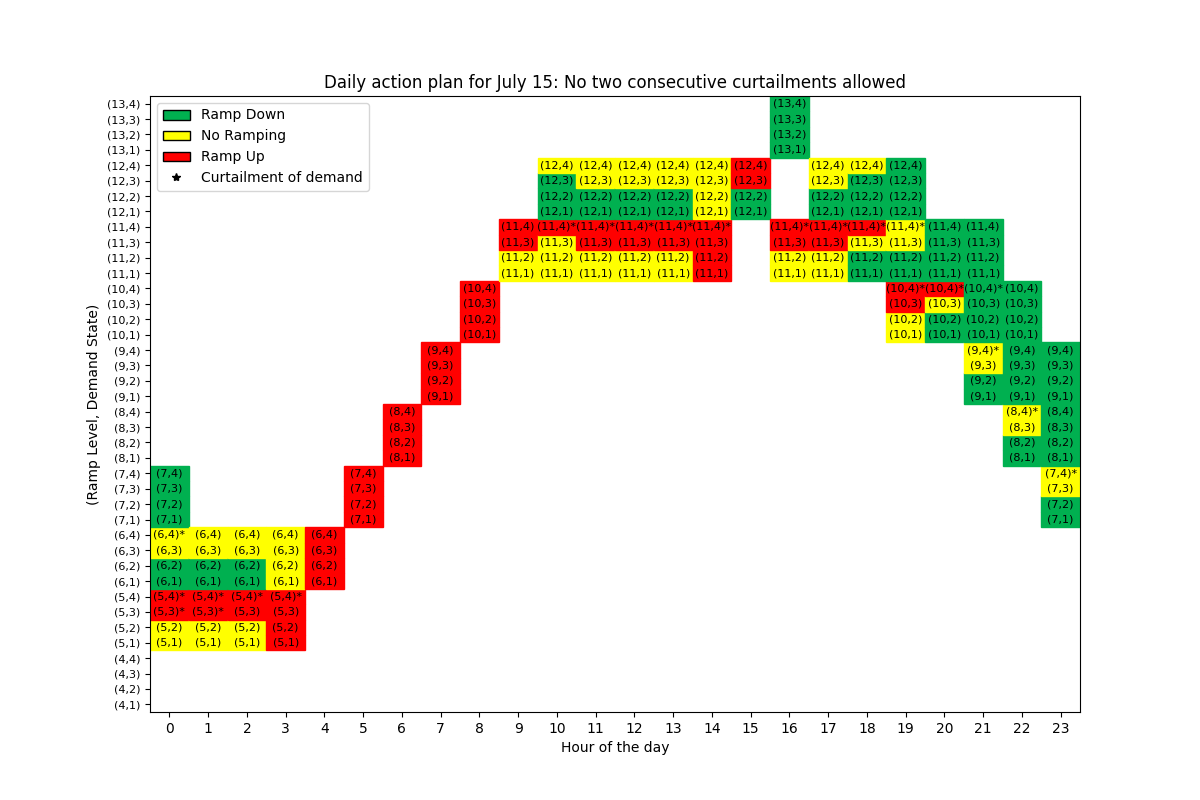}
        \end{minipage}
    
        \begin{minipage}{.49\textwidth}
            \centering
            \includegraphics[width=1\textwidth]{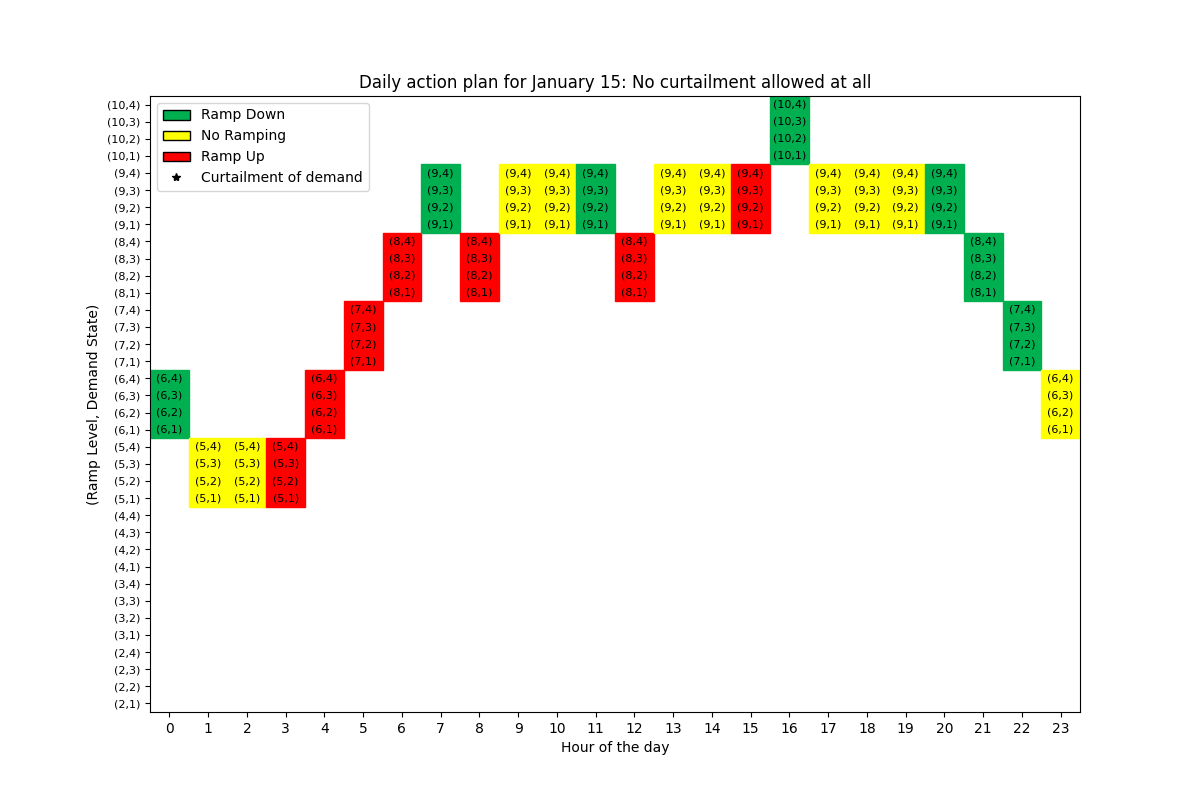}
        \end{minipage}%
        \begin{minipage}{.49\textwidth}
            \centering
            \includegraphics[width=1\textwidth]{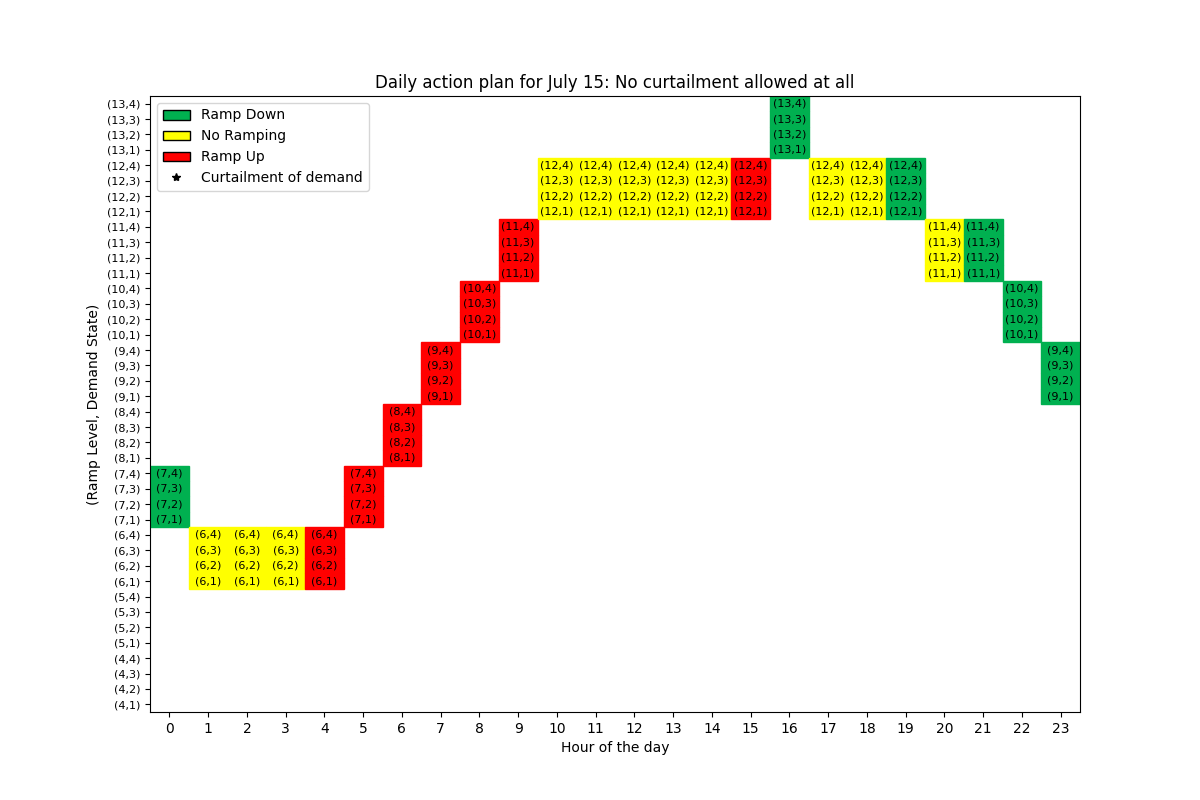}
        \end{minipage}
        
    \caption{Effect of limiting consecutive demand curtailment (Variation 3). Daily action plans for Jan. 15 (left) and Jul. 15 (right) under four scenarios (top to bottom): baseline (no consecutive-curtailment constraint), no three consecutive curtailment hours allowed, no two consecutive curtailment hours allowed, and no curtailment allowed at all.}
    \label{fig:ConsecutiveCurtailmentsTimeDependent}
\end{figure}
\clearpage



\color{lightgray} 

%

\color{black} 



{\scriptsize
        \begin{longtable}{c|cc|cc|cc|cc|cc|cc|cc}
            \caption{Comprehensive comparison of the CVaR approach and chance constraint variations with the risk-neutral approach}                                                                                                                  \\
            \toprule
            \multirow{2}{*}{Year}
                 & \multicolumn{2}{c}{Base LP: Risk-Neutral}
                 & \multicolumn{2}{c}{CVaR: Beta = 0.9}
                 & \multicolumn{2}{c}{CVaR: Beta = 0.99}
                 & \multicolumn{2}{c}{Var 1: $\tau=43.8$}
                 & \multicolumn{2}{c}{Var 2: $\mathcal{P}=0.005$}
                 & \multicolumn{2}{c}{Var 3: 3 CC NA}
                 & \multicolumn{2}{c}{Var 3: 2 CC NA}                                                                                                                                                      \\
            \cmidrule(r){2-3} \cmidrule(r){4-5} \cmidrule(r){6-7} \cmidrule(r){8-9} \cmidrule(r){10-11} \cmidrule(r){12-13} \cmidrule(r){14-15}
                 & CD                                             & Cost             & CD  & Cost   & CD  & Cost   & CD  & Cost   & CD    & Cost   & CD    & Cost   & CD  & Cost                           \\
            \midrule
            \endfirsthead
            \toprule
            \multirow{2}{*}{Year}
                 & \multicolumn{2}{c}{Base LP: Risk-Neutral}
                 & \multicolumn{2}{c}{CVaR: Beta = 0.9}
                 & \multicolumn{2}{c}{CVaR: Beta = 0.99}
                 & \multicolumn{2}{c}{Var 1: $\tau=43.8$}
                 & \multicolumn{2}{c}{Var 2: $\mathcal{P}=0.005$}
                 & \multicolumn{2}{c}{Var 3: 3 CC NA}
                 & \multicolumn{2}{c}{Var 3: 2 CC NA}                                                                                                                                                      \\
            \cmidrule(r){2-3} \cmidrule(r){4-5} \cmidrule(r){6-7} \cmidrule(r){8-9} \cmidrule(r){10-11} \cmidrule(r){12-13} \cmidrule(r){14-15}
                 & CD                                             & Cost             & CD  & Cost   & CD  & Cost   & CD  & Cost   & CD    & Cost   & CD    & Cost   & CD  & Cost                           \\
            \midrule
            \endhead
            \midrule
            \endfoot
            \bottomrule
            \multicolumn{15}{l}{\footnotesize \textbf{Abbreviations:}}                                                                                                                                     \\
            \multicolumn{15}{l}{\footnotesize \textbf{Var:} Variation, \quad  \textbf{CD:} Curtailed Demand (in \,GW), \quad \textbf{Cost:} Total cost of operation over the planning horizon.}            \\
            \multicolumn{15}{l}{\footnotesize \textbf{Base LP:} Risk-neutral baseline MDPLP~\eqref{eq:MDPLP} minimizing the expected cost.}                                                      \\
            \multicolumn{15}{l}{\footnotesize \textbf{Beta:} CVaR confidence level; hedging against the worst $(1 - \beta)$ tail of the cost distribution. Higher $\beta$ implies stronger risk aversion.} \\
            \multicolumn{15}{l}{\footnotesize \textbf{Var 1:} MDPLP~\eqref{eq:MDPLP} with \textit{$\tau = 43.8 \,GW$} as an annual average curtailed demand cap.}                                                               \\
            \multicolumn{15}{l}{\footnotesize \textbf{Var 2:} MDPLP~\eqref{eq:MDPLP} with probabilistic curtailment limit $\mathcal{P}_{year}=0.005$ as \eqref{eq:Var2}.}       \\
            \multicolumn{15}{l}{\footnotesize \textbf{Var 3:} MDPLP~\eqref{eq:MDPLP} preventing consecutive curtailments of electricity demand with:}            \\
            \multicolumn{15}{l}{\footnotesize \quad \textit{3 CC NA:} Three consecutive curtailments or more are not allowed. \quad \textit{2 CC NA:} Two consecutive curtailments or more are not allowed.} 
            \\
            \multicolumn{15}{l}{\footnotesize Year 2014 is omitted because the available data are not comparable.}

            \endlastfoot

            2006 & 710                                            & 12,607           & 485 & 13,757 & 202 & 16,015 & 581 & 12,655 & 574   & 12,551 & 673   & 12,722 & 617 & 12,887                         \\
            2007 & 178                                            & 3,696            & 73  & 4,096  & 13  & 4,798  & 125 & 3,707  & 125   & 3,668  & 156   & 3,731  & 133 & 3,805                          \\
            2008 & 453                                            & 10,095           & 259 & 10,919 & 73  & 12,703 & 362 & 10,179 & 359   & 10,122 & 425   & 10,180 & 383 & 10,336                         \\
            2009 & 248                                            & 10,089           & 167 & 11,522 & 65  & 14,021 & 192 & 10,422 & 194   & 10,300 & 242   & 10,288 & 203 & 10,553                         \\
            2010 & 645                                            & 14,489           & 437 & 16,179 & 223 & 19,255 & 538 & 14,746 & 521   & 14,599 & 614   & 14,694 & 556 & 14,991                         \\
            2011 & 267                                            & 8,257            & 168 & 9,045  & 72  & 10,967 & 203 & 8,358  & 212   & 8,310  & 247   & 8,317  & 215 & 8,430                          \\
            2012 & 387                                            & 12,010           & 246 & 13,259 & 94  & 16,082 & 318 & 12,247 & 322   & 12,177 & 374   & 12,155 & 345 & 12,412                         \\
            2013 & 378                                            & 11,157           & 253 & 12,391 & 149 & 14,993 & 297 & 11,387 & 300   & 11,284 & 352   & 11,274 & 319 & 11,522                         \\
            2015 & 49                                             & 3,508            & 17  & 4,321  & 0   & 5,410  & 28  & 3,622  & 30    & 3,618  & 37    & 3,592  & 25  & 3,710                          \\
            2016 & 274                                            & 13,245           & 173 & 15,309 & 49  & 18,671 & 217 & 13,687 & 217   & 13,564 & 256   & 13,501 & 226 & 13,909                         \\
            2017 & 148                                            & 11,724           & 70  & 13,871 & 11  & 17,240 & 96  & 12,183 & 102   & 12,062 & 128   & 11,956 & 107 & 12,389                         \\
            2018 & 296                                            & 13,181           & 202 & 15,374 & 105 & 18,773 & 236 & 13,667 & 237   & 13,515 & 282   & 13,470 & 256 & 13,912                         \\
            2019 & 113                                            & 12,576           & 41  & 14,792 & 16  & 18,419 & 65  & 13,115 & 67    & 12,969 & 100   & 12,836 & 68  & 13,271                         \\
            2020 & 105                                            & 11,948           & 44  & 14,104 & 29  & 17,449 & 64  & 12,461 & 61    & 12,324 & 82    & 12,163 & 65  & 12,611                         \\
            \hline
            2021 & 1,749                                          & 16,151           & 170 & 14,884 & 25  & 18,160 & 665 & 14,001 & 1,009 & 14,721 & 1,396 & 15,549 & 513 & 13,917                         \\
            2022 & 997                                            & 14,144           & 86  & 14,959 & 6   & 18,522 & 276 & 13,110 & 454   & 13,319 & 765   & 13,918 & 222 & 13,325                         \\
            2023 & 493                                            & 8,213            & 37  & 8,748  & 0   & 10,932 & 109 & 7,724  & 212   & 7,792  & 360   & 8,059  & 77  & 7,807                          \\
            \hline
            \hline
            \textbf{Total}
                 & \textbf{7,490}                                 & \textbf{187,090}
                 & \textbf{2,928}                                 & \textbf{207,530}
                 & \textbf{1,132}                                 & \textbf{252,410}
                 & \textbf{4,372}                                 & \textbf{187,271}
                 & \textbf{4,996}                                 & \textbf{186,895}
                 & \textbf{6,489}                                 & \textbf{188,405}
                 & \textbf{4,330}                                 & \textbf{189,787}                                                                                                                       \\
            \label{tab:all_variations_concise_updated}
        \end{longtable}
}

Imposing chance constraints systematically shifts the action plans toward more conservative operation, with earlier ramp-ups and fewer sustained ramp-down periods during critical hours. These patterns align with the cost–curtailment comparisons in Table~\ref{tab:all_variations_concise_updated}.

\section*{Acknowledgment}
\vspace{-0.2em}
This material is based upon work supported by the U.S. Department of Energy’s Office of Energy Efficiency and Renewable Energy (EERE) under the Wind Energy Technologies Office (WETO) Award Number DE-EE0011269, the Massachusetts Clean Energy Center and the Maryland Energy Administration. The views expressed herein do not necessarily represent the views of the U.S. Department of Energy, the United States Government, the Massachusetts Clean Energy Center or the Maryland Energy Administration.

\clearpage
\FloatBarrier
\bibliographystyle{apalike-ejor}
\bibliography{references}

\end{document}